\documentclass[preprint,12pt]{elsarticle}
\mathchardef\mhyphen="2D
\usepackage{amssymb}
\usepackage{algcompatible}
\usepackage{algpseudocode}
\usepackage{multirow}
\usepackage{xspace}
\usepackage{hyperref}
\usepackage{listings}
\usepackage{url}
\usepackage{csquotes}
\usepackage{amsfonts,epsfig}
\usepackage{enumerate}
\usepackage{amsmath,amscd,amsfonts,amssymb}

\usepackage{lineno,hyperref}
\usepackage{graphicx,subfigure}
\usepackage{enumerate}
\usepackage{times}
\usepackage{extarrows}
\usepackage{stmaryrd}
\usepackage{textcomp}
\usepackage[mathscr]{euscript}
\usepackage{algorithm2e}
\usepackage{colortbl}
\usepackage{color,soul}
\usepackage{newunicodechar}
\usepackage{amsthm}
\theoremstyle{definition}
\newtheorem{thm}{Theorem}
\newtheorem{cor}[thm]{Corollary}
\newtheorem{lem}[thm]{Lemma}
\newtheorem{defn}{Definition}%
\newtheorem{exa}{Example}

\newcommand{\X}{{\bf X}}

\newcommand{\E}{\mathbb{E}}

\newcommand{\Ran}{\mathrm{Ran}}
\newcommand{\Proj}{\mathbf{P}}

\begin{document}

\begin{frontmatter}

\title{Pass-efficient Randomized Algorithms for Low-rank
Approximation of Quaternion Matrices}

\author[inst1]{Salman Ahmadi-Asl$^{*}$}

\affiliation[inst1]{organization={Research Center of the Artificial Intelligence Institute, and Lab of Machine Learning
and Knowledge Representation, Innopolis University, 420500,},
            city={Innopolis}, 
            country={Russia, {\em s.ahmadiasl@innopolis.ru}}}

\author[label1]{Malihe Nobakht Kooshkghazi$^{\dagger}$}
\affiliation[label1]{organization={Independent researcher, m.nobakht88@gmail.com},
            }

\author[inst2]{Valentin Leplat$^{\dagger}$}

\affiliation[inst]{organization={ Innopolis University, 420500, Russia, v.leplat@innopolis.ru\\ $^*$Corresponding author,\,\,$^{\dagger}$Equal contribution},
           }



    

\begin{abstract}
Randomized algorithms for low-rank approximation of quaternion matrices have gained increasing attention in recent years. However, existing methods overlook pass efficiency—the ability to limit the number of passes over the input matrix—which is critical in modern computing environments dominated by communication costs. We address this gap by proposing a suite of pass-efficient randomized algorithms that let users directly trade pass budget for approximation accuracy. Our contributions include: (i) a family of arbitrary-pass randomized algorithms for low-rank approximation of quaternion matrices that operate under a user-specified number of matrix views, and (ii) a pass-efficient extension of block Krylov subspace methods that accelerates convergence for matrices with slowly decaying spectra. Furthermore, we establish spectral norm error bounds showing that the expected approximation error decays exponentially with the number of passes. Finally, we validate our framework through extensive numerical experiments and demonstrate its practical relevance across multiple applications, including quaternionic data compression, matrix completion, image super-resolution, and deep learning.
\end{abstract}



\end{frontmatter}


\section{Introduction}
\label{sec:intro}
Quaternions are a four-dimensional extension of the concept of complex numbers. They were first introduced by the Irish mathematician William Rowan Hamilton in 1843 \cite{hamilton1848xi,rosenfeld2012history}. Unlike complex numbers, which consist of a real part and an imaginary part, quaternions have three imaginary components in addition to the real part.  They are important tools in several applications. For example, quaternions are used to represent the orientation of objects in 3D space, making them valuable for applications in robotics and control systems. Their ability to avoid singularities and compactly represent rotations makes them useful for describing the orientation of robotic bodies and devices \cite{vince2011quaternions}. In signal processing \cite{ell2014quaternion} and control systems, quaternions are utilized to express and manipulate 3D rotations and orientations. They offer advantages in terms of stability and computational efficiency. 
Other applications include quantum physics and quantum mechanics, 
where quaternions can be used to represent quantum states and perform 
operations on quantum systems \cite{finkelstein1962foundations}.

In the past few years, quaternion matrices \cite{zhang1997quaternions} have attracted increasing attention, 
and many properties of real and complex matrices, such as eigenvalues, the SVD \cite{le2004singular,le2007jacobi}, 
and the QR decomposition \cite{bunse1989quaternion}, have been studied in the quaternion setting. 
We briefly motivate such extensions.

A color image has three Red-Green-Blue (RGB) channels, so it can be represented as a pure quaternion-valued matrix. 
This representation allows us to treat the image as a single object instead of processing each channel separately.
For example, the quaternion SVD (QSVD) can be directly applied to the pure quaternion, and a low-rank approximation of it can be computed to compress the image.  

Recently, some of these developments have also been generalized to split-quaternions \cite{ablamowicz2020moore} 
and to Clifford algebras \cite{cao2022moore}.
Many iterative algorithms have been proposed to solve linear matrix equations over the quaternion field \cite{ahmadi2017iterative,ahmadi2017efficient}. 
Using efficient BLAS-3 matrix operations, structure-preserving real representations of quaternion matrices 
have proved to be a very effective approach for quaternion computations \cite{jia2019lanczos}.
{These efficient algorithms have been further improved by incorporating the randomization framework \cite{liu2022randomized,ren2022randomized}.} Randomized matrix approximation is a technique used to approximate large matrices by constructing a low-rank approximation using random sampling. This approach is particularly useful for handling large-scale data, where traditional deterministic methods are computationally expensive. It is widely used in various applications, including data compression and dimensionality reduction, solving large-scale linear systems and eigenvalue problems, and machine learning and data analysis, such as in principal component analysis (PCA) and recommendation systems. {The works \cite{liu2022randomized,ren2022randomized} propose efficient randomized} algorithms based on random projection for the QSVD and QLP decompositions. However, they do not focus on pass efficiency, which is a main bottleneck on modern computer architectures. Indeed, pass-efficient algorithms refer to those that minimize the number of passes over the data. This is particularly important for large datasets where reading the entire dataset multiple times can be computationally expensive. 

Motivated by this issue, we propose pass-efficient randomized algorithms for low-rank approximation of quaternion matrices.

Our main contributions can be summarized as follows:
\begin{itemize}
    \item We propose a suite of pass-efficient randomized algorithms for low-rank approximation of quaternion matrices, enabling {explicit} control over the number of data passes while maintaining high approximation quality.

    \item We introduce a pass-efficient extension to block Krylov subspace methods in the quaternion setting, which improves convergence for matrices with slowly decaying singular values.

    \item We provide a detailed theoretical analysis of the proposed algorithms, including spectral norm bounds and structural insights into the quaternionic case.

    \item We demonstrate the practical relevance of our methods through extensive experiments on quaternionic data, with applications in color image compression, matrix completion, super-resolution, and deep learning.
\end{itemize}

The paper is structured as follows: In Section \ref{sec:sample:appendix}, necessary concepts and mathematical formulas are presented. In Section \ref{sec:propo}, we present our proposed pass-efficient randomized algorithm with a detailed discussion on its theoretical and numerical properties. The theoretical results are given in Section \ref{sec:the}. We introduce a pass-efficient extension to block Krylov subspace in Section~\ref{sec:BKrylov}. The computational complexity of the algorithms are presented in Section \ref{Sec:compcost}.
Two applications of the proposed randomized pass-efficient algorithm in image completion and image super-resolution are {presented} in Section \ref{Sec:appl}. Section \ref{sec:experi} is devoted to the experimental results, and we give a conclusion in Section \ref{sec:conclu}.

In the next section, we present the necessary mathematical preliminaries.

\section{Preliminaries}
\label{sec:sample:appendix}
We denote matrices by capital bold letters, e.g. $\X$. 
We denote the spectral norm of a matrix by $\|\cdot\|_2$, 
the Frobenius norm by $\|\cdot\|_F$, 
and the Euclidean norm of a vector by $\|\cdot\|$. 
The expectation operator is denoted by $\mathbb{E}[\cdot]$. For a matrix $\bf X$, $\operatorname{Ran}(\bf X)$ and $\bf X^\dagger$ denote its range (column space) and Moore-Penrose inverse, respectively. The notation ${ \bf W } \stackrel{d}{=}{\bf \Omega}$, means that the random matrix (or random vector) $\mathbf{W}$ is equal in distribution to the random matrix $\mathbf{\Omega}$

A quaternion is typically represented as:
\[
\mathbf{q} = q_0 + q_1 \mathbf{i} + q_2 \mathbf{j} + q_3 \mathbf{k}
\]
where \( q_0, q_1, q_2, q_3 \) are real numbers, and \( \mathbf{i}, \mathbf{j}, \mathbf{k} \) are the fundamental quaternion units satisfying the following multiplication rules:
\[
\mathbf{i}^2 = \mathbf{j}^2 = \mathbf{k}^2 = \mathbf{i}\mathbf{j}\mathbf{k} = -1
\]
The sum of two quaternions \( \mathbf{q} = q_0 + q_1 \mathbf{i} + q_2 \mathbf{j} + q_3 \mathbf{k} \) and \( \mathbf{p} = p_0 + p_1 \mathbf{i} + p_2 \mathbf{j} + p_3 \mathbf{k} \) is given by:
    \[
    \mathbf{q} + \mathbf{p} = (q_0 + p_0) + (q_1 + p_1) \mathbf{i} + (q_2 + p_2) \mathbf{j} + (q_3 + p_3) \mathbf{k}
    \]
The product of two quaternions \( \mathbf{q} \) and \( \mathbf{p} \) is given by:
\begin{eqnarray}
    \nonumber
    \mathbf{q} \mathbf{p} = (q_0 p_0 - q_1 p_1 - q_2 p_2 - q_3 p_3) + (q_0 p_1 + q_1 p_0 + q_2 p_3 - q_3 p_2)\\
     \mathbf{i} + (q_0 p_2 - q_1 p_3 + q_2 p_0 + q_3 p_1) \mathbf{j} + (q_0 p_3 + q_1 p_2 - q_2 p_1 + q_3 p_0) \mathbf{k}
\end{eqnarray}
The conjugate of a quaternion \( \mathbf{q} \) is defined as: $
    \mathbf{q}^* = q_0 - q_1 \mathbf{i} - q_2 \mathbf{j} - q_3 \mathbf{k}$.
The norm (or magnitude) of a quaternion \( \mathbf{q} \) is given by:$
    \|\mathbf{q}\| = \sqrt{q_0^2 + q_1^2 + q_2^2 + q_3^2},$
and the inverse of a non-zero quaternion \( \mathbf{q} \) is given by: $
    \mathbf{q}^{-1} = \frac{\mathbf{q}^*}{\|\mathbf{q}\|^2}.$
    
A quaternion \( \mathbf{q} \) is called a unit quaternion if its norm is 1, i.e., \( \|\mathbf{q}\| = 1 \). 
It is interesting to note that a unit quaternion \( \mathbf{q} = q_0 + q_1 \mathbf{i} + q_2 \mathbf{j} + q_3 \mathbf{k} \) can represent a rotation in 3D space\footnote{{The unit quaternion $q$ representing a rotation of angle $\theta$ around the unit axis 
${\bf u}$ is defined as 
$q = \cos\!\big(\tfrac{\theta}{2}\big) + {\bf u}\,\sin\!\big(\tfrac{\theta}{2}\big)$, 
where ${\bf u}$ is the normalized rotation axis.}}. Indeed, given a vector \( \mathbf{v} = v_1 \mathbf{i} + v_2 \mathbf{j} + v_3 \mathbf{k} \), the rotated vector \( \mathbf{v}' \) is given by:
    \[
    \mathbf{v}' = \mathbf{q} \mathbf{v} \mathbf{q}^*,
    \]
where \( \mathbf{v} \) is treated as a pure quaternion (i.e., with \( q_0 = 0 \)).

A quaternion matrix is a matrix whose elements are quaternions. The space of quaternion matrices of size $m\times n$ is denoted by $\mathbb{Q}^{m\times n}$. The conjugate transpose of a quaternion
matrix ${\bf X}={\bf X}_1+{\bf X}_2{\bf i}+{\bf X}_3{\bf j}+{\bf X}_4{\bf k}$ is defined as 
${\bf X}^H={\bf X}^T_1-{\bf X}^T_2{\bf i}-{\bf X}^T_3{\bf j}-{\bf X}^T_4{\bf k}$,
and its Frobenius norm is defined as
$
||{\bf X}||_F=\sqrt{\sum_{i,j}||x_{i,j}||^2}.
$ A quaternion matrix ${\bf V} \in \mathbb{Q}^{n \times n}$ is said to be unitary if ${\bf V}^H {\bf V} = {\bf I}$, where ${{\bf V}^H}$ denotes the {conjugate transpose} of {$\bf V$}.

\begin{defn}
(Quaternion Singular Value Decomposition (QSVD) \cite{zhang1997quaternions} and Quaternion QR decomposition \cite{bunse1989quaternion})  
Let ${\bf X}\in\mathbb{Q}^{m\times n}$ be a quaternion matrix; then it can be decomposed in the following forms:
\begin{eqnarray}\label{qsvd}
{\bf X}&=&{\bf U}\,{\bf \Sigma}\,{\bf V}^H,\\
\label{QQR}
{\bf X}&=&{\bf Q}\,{\bf R}.
\end{eqnarray}
The decomposition in \eqref{qsvd}, is called quaternion SVD (QSVD), where ${\bf U}\in\mathbb{Q}^{m\times m}$ and ${\bf V}\in\mathbb{Q}^{n\times n}$ are unitary quaternion matrices, while ${\bf \Sigma}\in\mathbb{R}^{m\times n}$ is a {real} diagonal matrix with the following structure: 
\[
{\bf \Sigma}=\begin{bmatrix}
{\bf \Sigma}_r & {\bf 0}\\
{\bf 0} & {\bf 0}
\end{bmatrix}\in\mathbb{R}^{m\times n}
\]
in which ${\bf \Sigma}_r={\rm diag}(\sigma_1,\sigma_2,\ldots,\sigma_r)\in\mathbb{R}^{r\times r}$, $r$ is the matrix rank of ${\bf X}$ and $\sigma_1\geq \sigma_2\geq\cdots\geq \sigma_r>0$. 
Also, the decomposition \eqref{QQR} is called quaternion QR decomposition, where ${\bf Q}\in\mathbb{Q}^{m\times n}$ {has orthonormal columns}
while ${\bf R}\in\mathbb{Q}^{n\times n}$ is an upper triangular matrix.
\end{defn}
 
\begin{defn}
({\bf Gaussian random matrix} \cite{liu2022randomized}) 
A Gaussian quaternion matrix is defined as 
\[
 {\bf \Omega}={\bf \Omega}_1+{\bf \Omega}_2{\bf i}+{\bf \Omega}_3{\bf j}+{\bf \Omega}_4{\bf k},   
\]
where ${\bf \Omega}_\ell\in\mathbb{R}^{m\times n}$, $\ell=1,\dots,4$, are independent real Gaussian matrices whose entries are i.i.d.\ ${\mathcal N}(0,1)$.
\end{defn}

With these foundations established, we introduce our pass-efficient algorithms.

\section{Proposed pass-efficient algorithms for quaternion matrices}\label{sec:propo}
This section is devoted to presenting our pass-efficient randomized algorithms for low-rank approximation of quaternion matrices. Randomization has been proved to be an efficient framework for performing a variety of linear and multilinear operations and decompositions. In particular, such algorithms are of interest due to their communication efficiency, robustness, lower computational complexity, probabilistic guarantees, and flexibility. In the era of big data, access to the full data matrix is often limited to a small number of passes—sometimes even a single pass—due to high communication costs. This challenges the applicability of traditional iterative algorithms, such as classical Krylov subspace methods, which typically require multiple passes. Motivated by this, we first develop pass-efficient techniques for low-rank approximation that work under strict pass constraints. In Section 5, we then show how these techniques can be extended to design pass-efficient variants of block Krylov subspace methods that retain low-pass complexity while accelerating convergence.

To the best of our knowledge, there are only a few papers on randomized algorithms for quaternion matrices \cite{liu2022randomized,ren2022randomized,liu2024fixed,chang2024one,yang2024randomized}, none of them discusses the pass efficiency of algorithms using an arbitrary number of passes. Besides, we exploit the proposed pass-efficient algorithms for the task of color image inpainting, compression, super-resolution, and deep learning as a new application and contribution.

{We first briefly review randomized algorithms for low-rank matrix approximation.}
Given a matrix ${\bf X} \in \mathbb{Q}^{m \times n}$, the goal is to find a low-rank approximation $\tilde{\bf X}$  such that:
\[
{\bf X} \approx \tilde{\bf X} = {\bf Q} {\bf Z},
\]
where \( {\bf Q} \in \mathbb{Q}^{m \times k} \) is an orthonormal matrix and \( {\bf Z} \in \mathbb{Q}^{k \times n} \) is a smaller matrix, with \( k \ll \min(m, n) \).
The randomized algorithm for low-rank matrix approximation typically involves the following steps:

\begin{enumerate}
    \item \textbf{Sampling}: Generate a random matrix \( {\bf \Omega} \in \mathbb{Q}^{n \times k} \) with entries drawn from a suitable distribution (e.g., Gaussian).
    \item \textbf{Projection}: Compute the matrix \( {\bf K} = {\bf X \Omega} \).
    \item \textbf{Orthonormalization}: Compute an orthonormal basis \( {\bf Q} \) for the range of \( {\bf K} \) using QR decomposition:
    \[
    {\bf K} = {\bf Q R}.
    \]
    \item \textbf{Approximation}: Form the low-rank approximation
    \[
    \tilde{\bf X} = {\bf Q Q}^H {\bf X}.
    \]
\end{enumerate}

The approximation described above can be improved by using the oversampling concept, which means the sampling matrix ${\bf \Omega} \in \mathbb{Q}^{n \times (k+p)}$ instead of a sampling matrix of size $n\times k$. This helps to better capture the range of the matrix ${\bf X}$. However, this approach is applicable only when the singular values of the matrix ${\bf X}$ decrease quite rapidly. To handle the scenario of matrices with low decaying singular values, the power iteration scheme is used.  

{Power iteration can be used to refine these approximations by improving the alignment with the dominant singular vectors. 
In the power scheme one forms
\[
{\bf K} = ({\bf X}{\bf X }^H)^q{\bf X} {\bf \Omega},
\]
where $q$ is the number of power iterations.} Increasing $q$ improves the accuracy of the approximation, while each power iteration requires additional matrix multiplications, increasing the computational cost. The randomized algorithm equipped with the oversampling and power iteration scheme is presented in Algorithm \ref{ALg_1}. An essential property of this algorithm is that for a given power iteration parameter $q$, the algorithm requires viewing or passing the data matrix ${\bf X},$ $2q+2$ times. More precisely, $2q$ times in lines 4-5 inside the ``for'' loop and two additional passes in lines 2 and 7. So, the question is whether it is possible to have a randomized algorithm that can provide a low-rank approximation for any budget of views, not necessarily even numbers. 

{This question was answered positively in \cite{bjarkason2019pass}, where a flexible algorithm was proposed for any number of matrix views. 
The same idea was later extended to tensors in \cite{ahmadi2024randomized}.}
Here, we use this technique for quaternion matrices. The algorithm is summarized in Algorithm \ref{ALg_2}. To be more precise, Algorithm \ref{ALg_2}, for a given budget $v\geq 2$ of views, computes an approximate truncated SVD $[{\bf U}_R,{\bf \Sigma}_R,{\bf V}_R]$ using one of the following strategies: 

\begin{itemize}
\item (If $v$ is even) Compute an orthonormal matrix ${\bf Q}^{(2)}$ whose columns build a basis for the
column space of $({\bf X}{\bf X}^H)^{(v-2)/2}{\bf X}{\bf \Omega}$. Then compute the rank-R TSVD $[{\bf V}_R,{\bf \Sigma}_R,\widetilde{{\bf U}_R}]$ of ${\bf X}^H{\bf Q}^{(2)}$ and set ${\bf U}_R={\bf Q}^{(2)}\widetilde{{\bf U}_R}$.

\item (If $v$ is odd) Compute an orthonormal matrix {${\bf Q}^{(1)}$} whose columns form a basis for the
column space of $({\bf X}^H{\bf X})^{(v-1)/2}{\bf \Omega}$. Then compute the rank-R TSVD $[{\bf U}_R,{\bf \Sigma}_R,\widetilde{{\bf V}_R}]$ of ${\bf X}{\bf Q}^{(1)}$  and set ${\bf V}_R={\bf Q}^{(1)}\widetilde{{\bf V}_R}$.
\end{itemize}

It can be easily seen that for $v=2(q+1)$, the first option is reduced to Algorithm \ref{ALg_1}. However, when we have a budget of an odd number of views, the second stage can be used. The second stage is a modification of the subspace algorithms proposed in \cite{vogel1994iterative,rokhlin2010randomized} as discussed in \cite{bjarkason2019pass}. Algorithm \ref{ALg_2} is more flexible in terms of the number of views/passes to get a low-rank matrix approximation and can be used for any number of passes greater than 2.

\RestyleAlgo{ruled}
\begin{algorithm}
{\smaller
\LinesNumbered
\SetKwInOut{Input}{Input}
\SetKwInOut{Output}{Output}  \Input{A data matrix ${\bf X}\in\mathbb{Q}^{I_1\times I_2}$; a target rank $k$; Oversampling $p$ and the power iteration $q$.}  \Output{Truncated SVD: ${\bf X}\cong \hat{{\bf X}}^{(q)}={\bf U}{\bf S}{\bf V}^H$}
\caption{Classical randomized subspace method for computation of the truncated SVD}\label{ALg_1}
${\bf \Omega}$ is a quaternion Gaussian random matrix of size $I_2\times (p+k)$;\\

$[{\bf Q}^{(1)},\sim] = {{\rm QR}}({\bf X}{\bf \Omega})$;\\
\For{$i=1,2,\ldots,q$}{
$[{\bf Q}^{(2)},\sim] = {{\rm QR}}({\bf X}^H{\bf Q}^{(1)})$;\\
$[{\bf Q}^{(1)},\sim] = {{\rm QR}}({\bf X}{\bf Q}^{(2)})$;\\
}
$[{\bf Q}^{(2)},{\bf R}] = {{\rm QR}}({\bf X}^H{\bf Q}^{(1)})$;\\
$[\widehat{\bf V},{\bf S},\widehat{\bf U}] = {\rm Truncated }$ $\,{\rm SVD}({\bf R},k)$;\\
${\bf V}={\bf Q}^{(2)}\widehat{\bf V}$;\\
${\bf U}={\bf Q}^{(1)}\widehat{\bf U}$;
}
\end{algorithm}

\RestyleAlgo{ruled}
\begin{algorithm}
{\smaller
\LinesNumbered
\SetKwInOut{Input}{Input}
\SetKwInOut{Output}{Output}  \Input{A data matrix ${\bf X}\in\mathbb{Q}^{I_1\times I_2}$; a target rank $k$; Oversampling $p$ and the budget of number of passes $v$.} \Output{Truncated SVD: ${\bf X}\cong {\bf U}{\bf S}{\bf V}^H$}
\caption{Randomized truncated SVD with an arbitrary number of passes}\label{ALg_2}
${\bf Q}^{(2)}$ is a quaternion Gaussian random matrix of size $I_2\times (p+k)$;\\
\For{$i=1,2,\ldots,v$}{
\eIf{$i$ is odd}
{$[{\bf Q}^{(1)},{\bf R}^{(1)}] = {{\rm QR}}({\bf X}{\bf Q}^{(2)})$;}
{$[{\bf Q}^{(2)},{\bf R}^{(2)}] = {{\rm QR}}({\bf X}^H{\bf Q}^{(1)})$;}
}
\eIf{$v$ is even}
{$[\widehat{\bf V},{\bf S},\widehat{\bf U}]=$Truncated SVD$({\bf R}^{(2)},k)$;}
{$[\widehat{\bf U},{\bf S},\widehat{\bf V}]=$Truncated SVD$({\bf R}^{(1)},k)$;}
${\bf V}={\bf Q}^{(2)}\widehat{\bf V}$;\\
${\bf U}={\bf Q}^{(1)}\widehat{\bf U}$;
}
\end{algorithm}

The next section presents the theoretical results of the proposed algorithm. 

\section{Theoretical results}\label{sec:the}

{This section is devoted to computing upper bounds for the expected approximation error
\[
\|\hat{{\bf X}}^{(q)}_{k+p}-{\bf X}\|_{2}
= \| ({\bf I}-{\bf Q}^{(1)}{\bf Q}^{(1)H}) {\bf X}\|_{2},\qquad q\geq 0,
\]
of the approximations computed by the proposed algorithms.}
The following lemma presents an upper bound for the expected approximation error (in spectral norm) for a rank-$(k + p)$ matrix of Algorithm \ref{ALg_1}. 
\begin{lem}\label{mn1} \cite{liu2022randomized}
(Deviation bound for approximation errors of Algorithm~\ref{ALg_1}.) Let the QSVD of the $I_{1}\times I_{2},\,(I_{1}\geq I_{2})$ quaternion matrix ${\bf X}$ be 
\begin{eqnarray*}
{\bf X}={\bf U}{\bf \Sigma}{\bf V}^H={\bf U\begin{bmatrix}
 {\bf \Sigma}_1&{\bf 0}\\
 {\bf 0} & {\bf \Sigma}_2
\end{bmatrix}}\begin{bmatrix}
   {\bf V}^H_1\\
   {\bf V}^H_2
 \end{bmatrix}, \quad {\bf \Sigma}_1 \in \Bbb R^{k \times k}, \quad {\bf V_1} \in \Bbb Q^{I_{2} \times k},
\end{eqnarray*}
where the singular value matrix ${\bf \Sigma}={\rm diag}(\sigma_1,\sigma_2,\ldots,\sigma_{I_{2}})$ with $\sigma_1\geq \sigma_2 \geq \dots \geq \sigma_{I_{2}} \geq 0 $, $k$ is the target rank. For oversampling parameter $p \geq 1$, let $q=0$, $k+p \leq I_{2}$, and the sample matrix ${\bf K}_{0}= \bf X \, \bf \Omega$, where $\bf \Omega$ is an $I_{2} \times (k+p)$ quaternion random test matrix, and $\bf \Omega_{1}=\bf V_{1}^* \bf \Omega$ is assumed to have full row rank, then the expected spectral-norm approximation error for the rank-$(k+p)$ matrix $ \hat{\bf X}^{(0)}_{k+p} $ (power-free case) {satisfies} 
\begin{eqnarray}
\E( \parallel \hat{{\bf X}}^{(0)}_{k+p}-{\bf X}\parallel_{2})\leq \big(1+3\sqrt{\frac{k}{4p+2}}\big)\sigma_{k+1}+\frac{3e\sqrt{4k+4p+2}}{2p+2}(\sum \limits_{j>k}\sigma_{j}^{2})^{\frac{1}{2}}.
\end{eqnarray}
 If $ q>0$, then the expected approximation of the spectral error {satisfies} 
\[
\E( \parallel \hat{{\bf X}}^{(q)}_{k+p}-{\bf X}\parallel_{2})\leq\big( \big(1+3\sqrt{\frac{k}{4p+2}}\big)\sigma_{k+1}^{2q+1}+ \frac{3e\sqrt{4k+4p+2}}{2p+2}(\sum \limits_{j>k}\sigma_{j}^{2(2q+1)})^{\frac{1}{2}}\big)^{\frac{1}{2q+1}}.
\]
\end{lem}
To derive spectral norm guarantees in the quaternionic setting, we first establish a structural result relating the spectral norm of projections of ${\bf X}$ and ${\bf K}=({\bf X}^{H} {\bf X})^{q}$ on the range space of the matrix ${\bf K}$. This is captured in the following lemma.

\begin{lem}\label{mn3}
Let $ {\bf X} \in \Bbb Q^{I_{1} \times I_{2}} $ has nonnegative singular values $ \sigma_{i}$, in which $i=1,\dots, \min(I_{1},I_{2})$, $ k\geq 2 $ be the target rank and $ p \geq 2 $ is an oversampling parameter, with $ k+p\leq \min(I_{1},I_{2}) $. Compose a Gaussian random matrix ${\bf \Omega} \in \Bbb Q^{I_{2} \times (k+p)}$ and put $ {\bf K}=({\bf X}^{H} {\bf X})^{q} {\bf \Omega} $, in which $ q\geq 1 $. Suppose that the orthonormal matrix $ {\bf Q}^{(2)} \in \Bbb Q^{I_{2} \times (k+p)}$ 
forms a basis for the range of $ {\bf K} $. Then 
\begin{eqnarray*}
\parallel {\bf X}({\bf I}-{\bf Q}^{(2)}{\bf Q}^{(2)H})\parallel_{2}^{2q}\leq \parallel ({\bf X}^{H}{\bf X})^{q}({\bf I}-{\bf Q}^{(2)}{\bf Q}^{(2)H}) \parallel_{2}.
\end{eqnarray*}
\end{lem}

\begin{proof}
Let 
\[
{\bf M} = {\bf I} - {\bf Q}^{(2)}{\bf Q}^{(2)H},
\]
so \({\bf M}^2={\bf M}={\bf M}^H\).  Write \({\bf X}={\bf U}{\bf \Sigma} {\bf V}^H\), hence \({\bf X}^H {\bf X}={\bf V}{\bf \Sigma}^2{\bf V}^H\), and set \({\bf A}={\bf X}^H {\bf X}\).  
By the Rayleigh quotient formula,
\[
\|{\bf X M}\|_2^2 
= \max_{\|{\bf u}\|=1} \|{\bf X M u}\|^2
= \max_{\|{\bf u}\|=1} {\bf u}^H {\bf M A M u}.
\]
Put \({\bf v}={\bf Mu}\).  Then \(\|{\bf v}\|\le1\) and \({\bf v}\in\mathrm{range}({\bf M})\), so
\[
\|{\bf X M}\|_2^2
= \max_{\substack{\|{\bf v}\|\le1 \\ v\in\mathrm{range}({\bf M})}} {\bf v}^H {\bf A v}.
\]
Raising to the \(q\)th power (\(q\ge1\)), monotonicity of \(t\mapsto t^q\) gives
\[
\|{\bf X M}\|_2^{2q}
= \Bigl(\max_{\substack{\|{\bf v}\|\le1 \\ {\bf v}\in\mathrm{range}({\bf M})}} {\bf v}^H {\bf A} {\bf v}\Bigr)^q
\le \max_{\substack{\|{\bf v}\|\le1 \\ {\bf v}\in\mathrm{range}({\bf M})}} ({\bf v}^H {\bf A} {\bf v})^q.
\]
By Jensen’s inequality (applied to the probability measure given by \(|v_i|^2\) in \({\bf A}\)’s eigenbasis),
\[
({\bf v}^H {\bf A v})^q \le {\bf v}^H {\bf A}^q {\bf v} 
\quad\text{for all }\|{\bf v}\|\le1.
\]
Thus
\[
\|{\bf X M}\|_2^{2q}
\le \max_{\substack{\|{\bf v}\|\le1 \\ v\in\mathrm{range}({\bf M})}} {\bf v}^H {\bf A}^q {\bf v}
= \max_{\substack{\|{\bf v}\|=1 \\ v\in\mathrm{range}({\bf M})}} {\bf v}^H {\bf A}^q {\bf v}.
\]
For any \({\bf v} \in \mathrm{range}({\bf M})\) with \(\|{\bf v}\|=1\), we have \({\bf Mv} = {\bf v}\). By the Cauchy-Schwarz inequality (valid in \(\mathbb{Q}\)),
\[
{\bf v}^H {\bf A}^q {\bf v} \leq \|{\bf v}\| \cdot \|{\bf A}^q {\bf v}\| = \|{\bf A}^q {\bf v}\|.
\]
Since \({\bf v} = {\bf Mv}\),
\[
\|{\bf A}^q {\bf v}\| = \|{\bf A}^q {\bf M v}\| \leq \|{\bf A}^q {\bf M}\|_2 \|{\bf v}\| = \|{\bf A}^q {\bf M}\|_2.
\]
Hence,
\[
{\bf v}^H {\bf A}^q {\bf v} \leq \|{\bf A}^q {\bf M}\|_2,
\]
and taking the maximum over \({\bf v}\),
\[
\max_{\substack{\|{\bf v}\|=1 \\ {\bf v}\in\mathrm{range}({\bf M})}} {\bf v}^H {\bf A}^q {\bf v} \leq \|{\bf A}^q {\bf M}\|_2 = \|({\bf X}^H {\bf X})^q {\bf M}\|_2.
\]
Therefore,
\[
\|{\bf X M}\|_2^{2q} \leq \|({\bf X}^H {\bf X})^q {\bf M}
\|_2.
\]
\end{proof}

For an even number of views, Algorithm \ref{ALg_2} is equivalent to Algorithm \ref{ALg_1} and its expected approximation error (in the spectral norm) can be approximated by Lemma \ref{mn1}. 
With Lemmas 1 and 2 in place, we are now ready to state our main theoretical result. Theorem 3 quantifies the spectral norm error of our pass-efficient algorithm \ref{ALg_2} in terms of an odd number of matrix views and the spectral decay of the input.

\begin{thm}\label{mn10}
With the notations in Lemma \ref{mn3}, the expected spectral norm of the approximation error of Algorithm \ref{ALg_2} satisﬁes
\begin{eqnarray*}
\E (\parallel {\bf X}-{\bf X}{\bf Q}^{(2)}{\bf Q}^{(2)H}\parallel_{2}) \leq \Big(\big(1+3\sqrt{\frac{k}{4p+2}}\big)\sigma_{k+1}^{2q}+\frac{3e\sqrt{4k+4p+2}}{2p+2}(\sum \limits_{j>k}\sigma_{j}^{4q})^{\frac{1}{2}} \Big)^{\frac{1}{2q}}.
\end{eqnarray*}
\end{thm}
\begin{proof}
Let $ {\bf Z}=({\bf X}^{H} {\bf X})^{q}$. The Hölder’s inequality and Lemma \ref{mn3} concludes that
\begin{eqnarray*}
\E (\parallel {\bf X}-{\bf X}{\bf Q}^{(2)}{\bf Q}^{(2)H}\parallel_{2})&\leq & (\E (\parallel {\bf X}-{\bf X}{\bf Q}^{(2)}{\bf Q}^{(2)H}\parallel_{2}^{2q}))^{\frac{1}{2q}}\\
&\leq & (\E(\parallel {\bf Z}-{\bf Z} {\bf Q}^{(2)}{\bf Q}^{(2)H}\parallel_{2}))^{\frac{1}{2q}}.
\end{eqnarray*}
Further, 
\begin{eqnarray*}
\nonumber {\E}(\parallel {\bf Z}-{\bf Z} {\bf Q}^{(2)}{\bf Q}^{(2)H}\parallel_{2})&=&\E(\parallel {\bf Z}^{H}-{\bf Q}^{(2)}{\bf Q}^{(2)H}  {\bf Z}^{H}\parallel_{2})\\
& \leq & \big (1+3\sqrt{\frac{k}{4p+2}}\big)\sigma_{k+1}^{2q}+\frac{3e\sqrt{4k+4p+2}}{2p+2}(\sum \limits_{j>k}\sigma_{j}^{4q})^{\frac{1}{2}},
\end{eqnarray*}
where $\sigma_{j}^{2q}$, for $ j=1,\dots,\min(I_{1},I_{2}) $, are the singular values of the matrix $ {\bf Z} $. The last inequality is obtained from applying Lemma \ref{mn1} in the case $ q=0 $ of the matrix $ {\bf Z} $.
\end{proof}

Theorem \ref{mn10} concludes that the expected error of the randomized approximation decays exponentially with 
the number of matrix views. For the accuracy of Algorithm 1,\ref{ALg_2} it is necessary that the spectrum of matrix $ {\bf X} $ decays rapidly, by Lemma \ref{mn1} and Theorem  \ref{mn10}. According to Lemma \ref{mn1} and Theorem \ref{mn10}, we deduce that if the number of matrix views $ v$ be is even, then Algorithm \ref{ALg_2} presents matrices $ {\bf Q}^{(1)} $ and $ {\bf Q}^{(2)} $ such that 
\begin{eqnarray}\label{mn11}
\nonumber & & \E(\parallel {\bf X}- {\bf Q}^{(1)} {\bf Q}^{(1)H} {\bf X}  \parallel_{2}) \leq \Big(\big(1+3\sqrt{\frac{k}{4p+2}}\big)\sigma_{k+1}^{v-1}+\frac{3e\sqrt{4k+4p+2}}{2p+2}(\sum \limits_{j>k}\sigma_{j}^{2(v-1)})^{\frac{1}{2}} \Big)^{\frac{1}{v-1}}.\\
& & \E(\parallel {\bf X}-  {\bf X}{\bf Q}^{(2)} {\bf Q}^{(2)H}   \parallel_{2}) \leq \Big(\big(1+3\sqrt{\frac{k}{4p+2}}\big)\sigma_{k+1}^{v}+\frac{3e\sqrt{4k+4p+2}}{2p+2}(\sum \limits_{j>k}\sigma_{j}^{2v})^{\frac{1}{2}} \Big)^{\frac{1}{v}}.
\end{eqnarray}
Also, If the number of matrix views $ v $ be an odd number, we have
\begin{eqnarray}\label{mn12}
& &\E(||{\bf X}- {\bf Q}^{(1)} {\bf Q}^{(1)H} {\bf X}||_{2}) \leq \Big(\big(1+3\sqrt{\frac{k}{4p+2}}\big)\sigma_{k+1}^{v}+\frac{3e\sqrt{4k+4p+2}}{2p+2}(\sum \limits_{j>k}\sigma_{j}^{2v})^{\frac{1}{2}} \Big)^{\frac{1}{v}}.\\
\nonumber & & \E(|| {\bf X}-  {\bf X}{\bf Q}^{(2)} {\bf Q}^{(2)H}||_{2}) \leq \Big(\big(1+3\sqrt{\frac{k}{4p+2}}\big)\sigma_{k+1}^{v-1}+\frac{3e\sqrt{4k+4p+2}}{2p+2}(\sum \limits_{j>k}\sigma_{j}^{2(v-1)})^{\frac{1}{2}} \Big)^{\frac{1}{v-1}}.
\end{eqnarray}
From ${\bf V}={\bf Q}^{(2)}\widehat{\bf V}$ and ${\bf U}={\bf Q}^{(1)}\widehat{\bf U}$, we deduce that the right and left singular vectors are as a linear combination of the columns ${\bf Q}^{(1)}$ and ${\bf Q}^{(2)}$, respectively. Moreover, relation \eqref{mn11} shows that for an even number of matrix views $ v $, $ {\bf V} $ can be more accurate than or as accurate as $ {\bf U} $. If $ v $ is an odd number, then relation \eqref{mn12} shows that $ {\bf U} $ is more accurate than or as accurate as $ {\bf V} $. Hence, for approximating the right singular vectors of matrix ${\bf X}$, when the number of views $v$ is an even number, it can be more accurate to apply Algorithm \ref{ALg_2} to ${\bf X}$. Also, when $ v $ is an odd number and our goal is to approximate the right singular vectors, we apply Algorithm \ref{ALg_2} to ${\bf X}^{H}$.

In the next section, we extend our approach to block Krylov subspace methods to improve convergence while preserving pass-efficiency.

\section{Proposed pass-efficient block Krylov algorithms for quaternion matrices}
\label{sec:BKrylov}

\providecommand{\E}{\mathbb{E}}
\providecommand{\Ran}{\mathrm{Ran}}
\providecommand{\Proj}{\mathbf{P}}

The randomized subspace iteration developed in Section~\ref{sec:propo} extends
naturally to \emph{block Krylov} methods, which are often observed to accelerate
convergence when the spectrum of ${\bf X}$ decays slowly. The key idea is to
enrich the sampled range by stacking multiple Krylov blocks (odd powers of
${\bf X}{\bf X}^H$ applied to a single random probe), while keeping the number
of matrix \emph{views} (multiplications by ${\bf X}$ or ${\bf X}^H$) under
explicit control.

\paragraph{Classical block Krylov scheme (even views).}
For a quaternion data matrix ${\bf X}\in\mathbb{Q}^{I_1\times I_2}$, oversampling
$p\ge1$, target rank $k$, and depth $q\ge0$, the standard block Krylov
construction forms
\[
\mathcal{K}_q({\bf X}{\bf X}^H,{\bf X}{\bf \Omega})
=\operatorname{span}\big\{{\bf X}{\bf \Omega},({\bf X}{\bf X}^H){\bf X}{\bf \Omega},\dots,({\bf X}{\bf X}^H)^q{\bf X}{\bf \Omega}\big\},
\]
where ${\bf \Omega}\in\mathbb{Q}^{I_2\times(k+p)}$ is quaternion Gaussian.  In
Algorithm~\ref{ALg_block}, forming ${\bf K}_0={\bf X}{\bf\Omega}$ costs one view,
each update ${\bf K}_i=({\bf X}{\bf X}^H){\bf K}_{i-1}$ costs two views, and the
compression step ${\bf Q}^H{\bf X}$ costs one more view. Hence the algorithm uses
exactly $2(q+1)$ views of ${\bf X}$.

\RestyleAlgo{ruled}
\begin{algorithm}
{\smaller
\LinesNumbered
\SetKwInOut{Input}{Input}
\SetKwInOut{Output}{Output}
\Input{${\bf X}\in\mathbb{Q}^{I_1\times I_2}$; target rank $k$; oversampling $p$; depth $q$.}
\Output{Truncated QSVD ${\bf X}\approx {\bf U}_{k}{\bf S}_{k}{\bf V}_{k}^H$.}
\caption{Block randomized subspace method (classical, even views)}
\label{ALg_block}
${\bf \Omega}\leftarrow \text{quaternion Gaussian of size } I_2\times(k+p)$\;
${\bf K}_0\leftarrow {\bf X}\,{\bf \Omega}$\;
\For{$i=1,\ldots,q$}{
  ${\bf K}_i \leftarrow ({\bf X}{\bf X}^H)\,{\bf K}_{i-1}$\;
}
${\bf K}\leftarrow[\,{\bf K}_0\ {\bf K}_1\ \cdots\ {\bf K}_q\,]$\;
$[{\bf Q},\sim]\leftarrow \mathrm{QR}({\bf K})$\;
$\tilde{\bf Y}\leftarrow {\bf Q}^H{\bf X}$\;
$[{\bf U}_{\tilde{Y}},{\bf S}_{\tilde{Y}},{\bf V}_{\tilde{Y}}]\leftarrow \mathrm{SVD}({\tilde{\bf Y}})$\;
$\hat{\bf U}\leftarrow {\bf Q}{\bf U}_{\tilde{Y}}$;\quad
${\bf U}_k\leftarrow \hat{\bf U}(:,1\!:\!k)$,\ 
${\bf S}_k\leftarrow {\bf S}_{\tilde{Y}}(1\!:\!k,1\!:\!k)$,\
${\bf V}_k\leftarrow {\bf V}_{\tilde{Y}}(:,1\!:\!k)$\;
}
\end{algorithm}

\paragraph{Flexible block Krylov (arbitrary views).}
In analogy with Algorithm~\ref{ALg_2} (Section~\ref{sec:propo}), we can aggregate
Krylov blocks to match any prescribed view budget $v\ge2$, building either a
range basis ${\bf Q}^{(1)}$ (even $v$) or a co-range basis ${\bf Q}^{(2)}$ (odd $v$).
Algorithm~\ref{ALg_block_flix} implements this strategy.  When $v=2(q+1)$ is
even, Algorithm~\ref{ALg_block_flix} reduces to Algorithm~\ref{ALg_block}.

\RestyleAlgo{ruled}

\begin{algorithm}
{\smaller
\LinesNumbered
\SetKwInOut{Input}{Input}
\SetKwInOut{Output}{Output}
\Input{${\bf X}\in\mathbb{Q}^{I_1\times I_2}$; target rank $k$; oversampling $p$; view budget $v\ge2$.}
\Output{Truncated QSVD ${\bf X}\approx {\bf U}{\bf S}{\bf V}^H$.}
\caption{Block randomized truncated SVD with an arbitrary number of views}
\label{ALg_block_flix}

${\bf \Omega}\leftarrow \text{quaternion Gaussian of size } I_2\times(k+p)$\;
${\bf Q}^{(2)}_0 \leftarrow {\bf \Omega}$\;
$s\leftarrow \lfloor (v-1)/2 \rfloor$\;

\For{$j=1,\ldots,s$}{
  \tcp{odd view: range side}
  \If{$2j-1 < v$}{
    $[{\bf Q}^{(1)}_j,\sim]\leftarrow \mathrm{QR}({\bf X}{\bf Q}^{(2)}_{j-1})$\;
  }
  \If{$2j < v$}{
    $[{\bf Q}^{(2)}_j,\sim]\leftarrow \mathrm{QR}({\bf X}^H{\bf Q}^{(1)}_{j})$\;
  }
}

\If{$v$ is even}{
  $[{\bf Q}^{(1)}_{s+1},\sim]\leftarrow \mathrm{QR}({\bf X}{\bf Q}^{(2)}_{s})$\;
}

\eIf{$v$ is even}{
  \tcp{range projector (even $v$): stack ${\bf Q}^{(1)}_1,\dots,{\bf Q}^{(1)}_{v/2}$}
  ${\bf K}_1\leftarrow [\,{\bf Q}^{(1)}_1\ {\bf Q}^{(1)}_2\ \cdots\ {\bf Q}^{(1)}_{v/2}\,]$\;
  $[{\bf Q}^{(1)},\sim]\leftarrow \mathrm{QR}({\bf K}_1)$\;
  $[{\bf Q}^{(2)},{\bf R}^{(2)}]\leftarrow \mathrm{QR}({\bf X}^H{\bf Q}^{(1)})$\; 
  $[\widehat{\bf V},{\bf S},\widehat{\bf U}]\leftarrow \mathrm{TSVD}({\bf R}^{(2)},k)$\;
}{
  \tcp{co-range projector (odd $v$): stack ${\bf Q}^{(2)}_1,\dots,{\bf Q}^{(2)}_{(v-1)/2}$}
  ${\bf K}_2\leftarrow [\,{\bf Q}^{(2)}_1\ {\bf Q}^{(2)}_2\ \cdots\ {\bf Q}^{(2)}_{(v-1)/2}\,]$\;
  $[{\bf Q}^{(2)},\sim]\leftarrow \mathrm{QR}({\bf K}_2)$\;
  $[{\bf Q}^{(1)},{\bf R}^{(1)}]\leftarrow \mathrm{QR}({\bf X}{\bf Q}^{(2)})$\; 
  $[\widehat{\bf U},{\bf S},\widehat{\bf V}]\leftarrow \mathrm{TSVD}({\bf R}^{(1)},k)$\;
}
${\bf V}\leftarrow {\bf Q}^{(2)}\widehat{\bf V}$,\quad
${\bf U}\leftarrow {\bf Q}^{(1)}\widehat{\bf U}$\;
}
\end{algorithm}

\subsection*{Corrected block-Krylov error bounds (no reverse-order law, no derived-Gaussian claims)}

The original quaternion block-Krylov analysis~\cite{qub} used a reverse-order
identity for the Moore--Penrose pseudoinverse at a key step of the proof of
their Theorem~4.2. That identity requires additional full-rank hypotheses which
are not satisfied by the block-stacked factors arising in the Krylov setting.
In particular, one cannot justify (in general) a step that implicitly treats a
derived stacked sampling factor as if it were an i.i.d.\ quaternion
Gaussian matrix.

We avoid this issue entirely. Our correction is based on two simple facts:
(i) the block Krylov subspace contains, as a subset, the single last
(power) sample; (ii) orthogonal projection error is monotone with respect to
subspace inclusion. This yields a simple reduction of block-Krylov error bounds
to the already-established quaternion randomized subspace-iteration bounds
(e.g.\ Lemma~\ref{mn1}).

\begin{lem}[Corrected block-Krylov bound via range inclusion]
\label{lem:corrected-block}
Let ${\bf X}\in\mathbb{Q}^{I_1\times I_2}$, target rank $k$, oversampling $p\ge1$,
and depth $q\ge0$. Run Algorithm~\ref{ALg_block} and let ${\bf Q}$ be the
orthonormal basis obtained from the stacked block Krylov matrix
\[
{\bf K}_{\rm stack}
=\big[\,{\bf X} {\bf \Omega},\ ( {\bf X}{\bf X}^H){\bf X}{\bf \Omega},\ \dots,\ ( {\bf X}{\bf X}^H)^q{\bf X}{\bf \Omega}\,\big],
\quad
{\bf \Omega}\in\mathbb{Q}^{I_2\times(k+p)}\ \text{quaternion Gaussian}.
\]
Define also the (single-block) power sample
\[
{\bf K}_{\rm pow}:=( {\bf X}{\bf X}^H)^q{\bf X}{\bf \Omega},
\]
and let  ${\bf Q}_{\rm pow}$ be an orthogonal matrix for the range of ${\bf K}_{\rm pow}$.

Then, deterministically,
\begin{equation}\label{eq:bkrylov-inclusion}
\|{\bf X}-{\bf Q}{\bf Q}^H{\bf X}\|_2 \ \le\ \|{\bf X}-{\bf Q}_{\rm pow}{\bf Q}_{\rm pow}^H{\bf X}\|_2.
\end{equation}
Consequently, using the standard quaternion subspace-iteration bound (Lemma~\ref{mn1}),
\begin{equation}\label{eq:BK-safe}
\E\|{\bf X}-{\bf Q}{\bf Q}^H{\bf X}\|_2
\ \le\
\Bigg(
\Big(1+3\sqrt{\tfrac{k}{\,4p+2\,}}\Big)\,\sigma_{k+1}^{\,2q+1}
\;+\;
\frac{3e\sqrt{\,4k+4p+2\,}}{\,2p+2\,}\,
\Big(\sum_{j>k}\sigma_j^{\,2(2q+1)}\Big)^{\!\frac12}
\Bigg)^{\!\frac1{2q+1}},
\end{equation}
where $\sigma_1\ge\sigma_2\ge\cdots$ are the singular values of ${\bf X}$.
\end{lem}

\begin{proof}
Since ${\bf K}_{\rm pow}$ is one block-column of ${\bf K}_{\rm stack}$, we have
$\Ran({\bf K}_{\rm pow})\subseteq \Ran({\bf K}_{\rm stack})$. Projection error is
monotone with respect to enlarging the projection subspace, hence
\eqref{eq:bkrylov-inclusion}. Taking expectation and applying Lemma~\ref{mn1} to
the power sample ${\bf K}_{\rm pow}=( {\bf X}{\bf X}^H)^q{\bf X}{\bf \Omega}$
yields \eqref{eq:BK-safe}.
\end{proof}

\begin{cor}[Block case $q=0$]\label{cor:q0}
With $q=0$ in Lemma~\ref{lem:corrected-block},
\[
\E\|{\bf X}-{\bf Q}{\bf Q}^H{\bf X}\|_2
\ \le\
\Big(1+3\sqrt{\tfrac{k}{\,4p+2\,}}\Big)\,\sigma_{k+1}
\;+\;
\frac{3e\sqrt{\,4k+4p+2\,}}{\,2p+2\,}\,
\Big(\sum_{j>k}\sigma_j^{\,2}\Big)^{\!\frac12}.
\]
\end{cor}

\medskip
\noindent\textbf{Approximation realized by Algorithm~\ref{ALg_block_flix}.}
If $v$ is \emph{even}, the algorithm returns a range projector and the approximation
is ${\bf X}\approx {\bf Q}^{(1)}{\bf Q}^{(1)H}{\bf X}$.
If $v$ is \emph{odd}, it returns a co-range projector and the approximation is
${\bf X}\approx {\bf X}{\bf Q}^{(2)}{\bf Q}^{(2)H}$.

\begin{thm}[Flexible block-Krylov bound for an arbitrary number of views]
\label{thm:flex-krylov}
Let ${\bf X}\in\mathbb{Q}^{I_1\times I_2}$ have QSVD ${\bf X}={\bf U}{ \bf\Sigma}{\bf V}^H$ with
singular values $\sigma_1\ge \sigma_2\ge\cdots$, target rank $k$, oversampling $p\ge1$,
and view budget $v\ge2$. Run Algorithm~\ref{ALg_block_flix} and let
${\bf Q}^{(1)}$ or ${\bf Q}^{(2)}$ be the basis it outputs according to the parity of $v$.
Define the exponent
\[
\alpha_v:=v-1.
\]

\textup{(a)} If $v$ is even, with ${\bf P}_v:={\bf Q}^{(1)}{\bf Q}^{(1)H}$,
\begin{equation}\label{eq:flex-even-safe}
\E\|{\bf X}-{\bf P}_v{\bf X}\|_2
\ \le\
\Bigg(
\Big(1+3\sqrt{\tfrac{k}{\,4p+2\,}}\Big)\,\sigma_{k+1}^{\,\alpha_v}
\;+\;
\frac{3e\sqrt{\,4k+4p+2\,}}{\,2p+2\,}\,
\Big(\sum_{j>k}\sigma_j^{\,2\alpha_v}\Big)^{\!\frac12}
\Bigg)^{\!\frac1{\alpha_v}}.
\end{equation}

\textup{(b)} If $v$ is odd, with ${\bf P}_v:={\bf Q}^{(2)}{\bf Q}^{(2)H}$,
\begin{equation}\label{eq:flex-odd-safe}
\E\|{\bf X}-{\bf X}{\bf P}_v\|_2
\ \le\
\Bigg(
\Big(1+3\sqrt{\tfrac{k}{\,4p+2\,}}\Big)\,\sigma_{k+1}^{\,\alpha_v}
\;+\;
\frac{3e\sqrt{\,4k+4p+2\,}}{\,2p+2\,}\,
\Big(\sum_{j>k}\sigma_j^{\,2\alpha_v}\Big)^{\!\frac12}
\Bigg)^{\!\frac1{\alpha_v}}.
\end{equation}
\end{thm}

\begin{proof}
\textbf{Even $v$ (range projector).}
Let $v=2(q+1)$, so $\alpha_v=2q+1$. By construction,
Algorithm~\ref{ALg_block_flix} forms and stacks Krylov blocks that include the
last power sample $( {\bf X}{\bf X}^H)^q{\bf X}{\bf \Omega}$ (up to intermediate
orthonormalizations which do not change the generated range). Hence the final
range basis ${\bf Q}^{(1)}$ spans a subspace that contains the range of the last
power sample. By monotonicity of orthogonal projection error with respect to
subspace inclusion, the approximation error is bounded by that of the standard
subspace-iteration/power scheme with exponent $2q+1=\alpha_v$. Applying Lemma~\ref{mn1}
yields \eqref{eq:flex-even-safe}.

\medskip
\textbf{Odd $v$ (co-range projector).}
Let $v=2s+1$, so $\alpha_v=2s$. Algorithm~\ref{ALg_block_flix} stacks co-range
iterates that include (up to orthonormalization) the last sample
\[
{\bf Y}:=( {\bf X}^H{\bf X})^{s}{\bf \Omega}\in\mathbb{Q}^{I_2\times(k+p)}.
\]
Let ${\bf Q}_{\rm last}$ span $\Ran({\bf Y})$ and set ${\bf P}_{\rm last}:={\bf Q}_{\rm last}{\bf Q}_{\rm last}^H$.
Since $\Ran({\bf Q}_{\rm last})\subseteq \Ran({\bf Q}^{(2)})$, again by monotonicity
\[
\|{\bf X}-{\bf X}{\bf P}_v\|_2 \ \le\ \|{\bf X}-{\bf X}{\bf P}_{\rm last}\|_2.
\]
Now set ${\bf M}:={\bf I}-{\bf P}_{\rm last}$ (orthogonal projector). A standard
deterministic reduction gives
\[
\|{\bf X}{\bf M}\|_2^{2s} \ \le\ \|({\bf X}^H{\bf X})^s{\bf M}\|_2.
\]
Because $({\bf X}^H{\bf X})^s$ is Hermitian, $\|({\bf X}^H{\bf X})^s{\bf M}\|_2=\|{\bf M}({\bf X}^H{\bf X})^s\|_2$,
hence
\[
\|{\bf X}-{\bf X}{\bf P}_{\rm last}\|_2^{\alpha_v}
=\|{\bf X}{\bf M}\|_2^{2s}
\ \le\
\|{\bf M}({\bf X}^H{\bf X})^s\|_2
=
\|({\bf X}^H{\bf X})^s-{\bf P}_{\rm last}({\bf X}^H{\bf X})^s\|_2.
\]
Taking expectations and Jensen (concavity of $t\mapsto t^{1/\alpha_v}$) yields
\[
\E\|{\bf X}-{\bf X}{\bf P}_{\rm last}\|_2
\ \le\
\Big(\E\|({\bf X}^H{\bf X})^s-{\bf P}_{\rm last}({\bf X}^H{\bf X})^s\|_2\Big)^{1/\alpha_v}.
\]
Finally apply the basic quaternion randomized range-finder bound (Lemma~\ref{mn1} with $q=0$)
to the matrix ${\bf A}:=({\bf X}^H{\bf X})^s$, whose singular values are
$\sigma_j({\bf A})=\sigma_j({\bf X})^{2s}=\sigma_j({\bf X})^{\alpha_v}$.
This gives \eqref{eq:flex-odd-safe}.
\end{proof}

\label{rem:flex-shapes}
\medskip
\noindent\textbf{Remark.}
The corrected bounds above are deliberately conservative: they certify that the
block Krylov spaces produced by Algorithms~\ref{ALg_block}--\ref{ALg_block_flix}
are at least as good as the corresponding last-iterate power samples, without
requiring any reverse-order pseudoinverse identity or any claim that a derived
stacked factor is quaternion Gaussian. In practice, stacking multiple Krylov
blocks typically improves performance beyond what is captured by these safe
bounds, which is consistent with the empirical observations in
Section~\ref{sec:experi}.
\smallskip

Sharper (less conservative) constants would require analyzing the actual
\emph{stacked} sampling factor that arises in the block-Krylov rearrangement:
it has deterministic zero off-diagonal subblocks and diagonal blocks that are
generally dependent, hence it is not an i.i.d.\ quaternion Gaussian matrix and
the usual rotational-invariance shortcuts do not apply. This issue is explicit
in the construction and block decomposition of the stacked factor; see
\cite[Eq.~(9), Eq.~(12)--(13)]{qub}. Developing such refined bounds
for dependent/structured Gaussian-like blocks is technically involved and is
left for future work.


\section{Execution cost}\label{Sec:compcost}

As highlighted in~\cite{bjarkason2019pass}, several practical considerations arise when applying Algorithm~\ref{ALg_2} to compute a low-rank approximation of a quaternion matrix.

First, the choice between applying Algorithm~\ref{ALg_2} to ${\bf X}$ or its conjugate transpose ${\bf X}^H$ depends on the matrix dimensions and pass count. For an \emph{even} number of views, the primary difference lies in the cost of generating the random Gaussian sampling matrix (Line 1 of Algorithm~\ref{ALg_2}), which amounts to $4I_2(p + k)T$, where $T$ denotes the cost of generating a single Gaussian random variable.\footnote{Faster random matrix generation schemes such as SRFT may be used, but they do not guarantee the error bounds established in Lemma~1 and Theorem~3.} To minimize computational overhead, ${\bf X}^H$ is preferred when ${\bf X}$ has more columns than rows; otherwise, using ${\bf X}$ is more efficient.

For an \emph{odd} number of views, the choice between ${\bf X}$ and ${\bf X}^H$ impacts not only the sampling cost but also the cost of additional QR decompositions and matrix multiplications. Specifically, the extra cost is on the order of
\[
\mathcal{O}\left(4I_2(p + k)T + (p + k)^2I_2 + (p + k)H\right),
\]
where $H$ represents the cost of multiplying the input matrix by a vector. Each term corresponds respectively to: (1) generating a random quaternion matrix, (2) performing a QR decomposition, and (3) a matrix-vector product.

Furthermore, it is straightforward to observe that Algorithm~\ref{ALg_block_flix} reduces to Algorithm~\ref{ALg_2} when $v < 4$. For larger values of $v$, however, the differences become more pronounced. Specifically, Algorithm~\ref{ALg_block_flix} multiplies the matrix ${\bf X}$ with $(v + \lfloor v/2 \rfloor - 1)(p + \ell)$ vectors and performs several QR decompositions, leading to an overall cost of
\[
\mathcal{O}\left(v^2 I_2 (p + \ell)^{2}\right).
\]
In contrast, Algorithm~\ref{ALg_2} performs $v(p + \ell)$ matrix-vector products and incurs a cost of
\[
\mathcal{O}\left(v I_2 (p + \ell)^2\right),
\]
assuming a square matrix size $I_1 = I_2$.

{In summary, Algorithm~\ref{ALg_block_flix} tends to have a higher computational cost per pass than Algorithm~\ref{ALg_2}. 
However, for $v \geq 4$, it typically achieves a given target accuracy using fewer passes, 
making it more efficient overall in practice for matrices with slowly decaying singular values.
}

\section{Pass-efficient quaternion randomization for fast image inpainting, image super-resolution and deep learning}\label{Sec:appl}

\paragraph{Quaternion Matrix Completion as a Core Tool}
Matrix completion aims to recover a low-rank matrix from a subset of its entries. It arises in diverse applications such as collaborative filtering, image inpainting, and super-resolution. Formally, let $\mathbf{M} \in \mathbb{Q}^{m \times n}$ denote the original low-rank matrix, observed only on a subset $\Omega$. The observed data is encoded via the projection operator:
\[
({\bf P}_{\Omega}({\bf{X}}))_{i,j}= 
\begin{cases}
{\bf X}_{ij}, & (i,j) \in \Omega \\
0, & \text{otherwise.}
\end{cases}
\]
The goal is to find a low-rank matrix $\mathbf{X}$ that agrees with $\mathbf{M}$ on the observed entries:
\begin{equation}
\min_{\mathbf{X}} \| \mathbf{P}_{\Omega}(\mathbf{X}) - \mathbf{P}_{\Omega}(\mathbf{M}) \|_F^2 \quad \text{s.t.} \quad \text{Rank}(\mathbf{X}) = R.
\label{MinRankCompl2}
\end{equation}
Using an auxiliary variable $\mathbf{C}$, this can be reformulated as:
\begin{equation}
\min_{\mathbf{X}, \mathbf{C}} \| \mathbf{X} - \mathbf{C} \|_F^2 \quad \text{s.t.} \quad
\begin{cases}
\text{Rank}(\mathbf{X}) = R \\
\mathbf{P}_{\Omega}(\mathbf{C}) = \mathbf{P}_{\Omega}(\mathbf{M})
\end{cases}
\label{MinRankCompl3}
\end{equation}
This problem can be solved via alternating updates:
\begin{align}
\mathbf{X}^{(n)} &\leftarrow \mathcal{L}(\mathbf{C}^{(n)}), \label{Step1} \\
\mathbf{C}^{(n+1)} &\leftarrow \mathbf{\Omega} \odot \mathbf{M} + (\mathbf{1} - \mathbf{\Omega}) \odot \mathbf{X}^{(n)}, \label{Step2}
\end{align}
where $\mathcal{L}$ denotes a low-rank approximation operator (e.g., our proposed randomized algorithm), $\odot$ is the Hadamard product, { and $\bf \Omega$ is the binary mask associated with the sampling set $\Omega$.}

\paragraph{Image Inpainting}
Image inpainting aims to reconstruct missing or corrupted regions of an image, restoring its visual coherence. This task is naturally formulated as a matrix completion problem when color image channels are modeled using quaternion-valued matrices. Our pass-efficient quaternion approximation methods are particularly suited for this task due to their low computational footprint.

\paragraph{Image Super-Resolution}
Super-resolution (SR) enhances the resolution of an image from a low-resolution input. By upsampling an image and artificially introducing missing rows and columns, the SR task becomes a structured inpainting problem. We apply the same quaternion matrix completion framework to recover the high-resolution image, effectively filling in the missing details using low-rank priors.

\paragraph{Improving Deep Learning Robustness}
As we demonstrate in Section~\ref{sec:experi}, many deep neural networks (DNNs) are highly sensitive to small input perturbations such as pixel dropouts or additive noise—an issue central to adversarial attacks~\cite{goodfellow2014explaining}. We show that applying our quaternion matrix completion as a preprocessing step improves DNN robustness under such perturbations, making this an effective defense strategy in sensitive applications.

\section{Experimental results}\label{sec:experi}

This section presents a series of numerical experiments that validate the effectiveness of the proposed pass-efficient quaternion algorithms across various tasks. All experiments were conducted in \textsc{Matlab} on a standard laptop equipped with an Intel(R) Core(TM) i7-10510U CPU and 16~GB of RAM. We evaluate three representative applications: (i) quaternionic image compression, (ii) image completion and super-resolution, and (iii) image segmentation, a key task in deep learning and computer vision. For simplicity, in all experiments, the random quaternion matrices used in Line~1 of Algorithms~\ref{ALg_1} and~\ref{ALg_2} were instantiated using real-valued Gaussian matrices.

\begin{exa}(Image compression)\label{exa_1} 
We evaluate the proposed algorithms on the Kodak image dataset\footnote{\url{http://www.cs.albany.edu/~xypan/research/snr/Kodak.html}}, selecting five images: ``Kodim13'', ``Kodim7'', ``Kodim17'', ``Kodim15'', and ``Kodim16''. All images are resized to $256 \times 256 \times 3$ for consistency and treated as pure quaternion-valued data.

Figure~\ref{benchmark} displays the selected test images. We apply the randomized low-rank approximation algorithms with rank $k = 30$, oversampling parameter $p = 5$, and varying numbers of passes. The reconstructed (compressed) images are shown in Figure~\ref{Res_1}, while Figure~\ref{Res_2} reports the corresponding CPU time and PSNR values.

As expected, increasing the number of passes improves reconstruction accuracy but also increases computation time. Notably, for the images tested, a three-pass randomized approximation achieves nearly the same PSNR as the four-pass version, while requiring less time. This highlights a key benefit of our pass-efficient framework: it enables users to balance approximation quality against computational budget, without committing to a fixed number of passes.

We also compare Algorithm~\ref{ALg_1} and the block-based Algorithm~\ref{ALg_block}. The latter yields slightly better approximation quality but incurs higher computational cost. Interestingly, its flexible variant, Algorithm~\ref{ALg_block_flix}, achieves a comparable approximation error with significantly reduced runtime. 

These results demonstrate that the proposed algorithms provide high-quality low-rank approximations with tunable efficiency, making them practical for real-world image compression tasks under varying resource constraints.

\begin{figure}
\begin{center}
\includegraphics[width=0.9\linewidth]{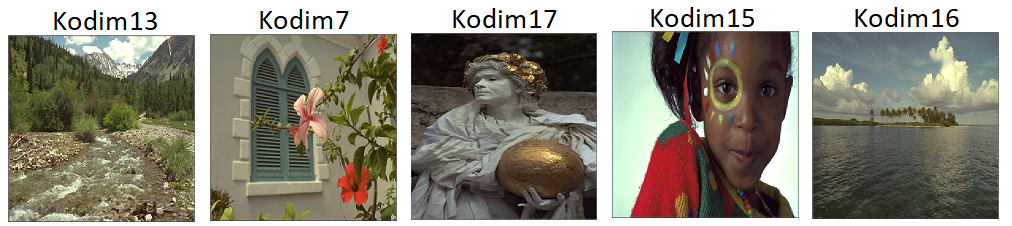}
\caption{\small{Benchmark images used in our simulation for image compression.}}\label{benchmark}
\end{center}
\end{figure}

\begin{figure}
\begin{center}
\includegraphics[width=0.9\linewidth]{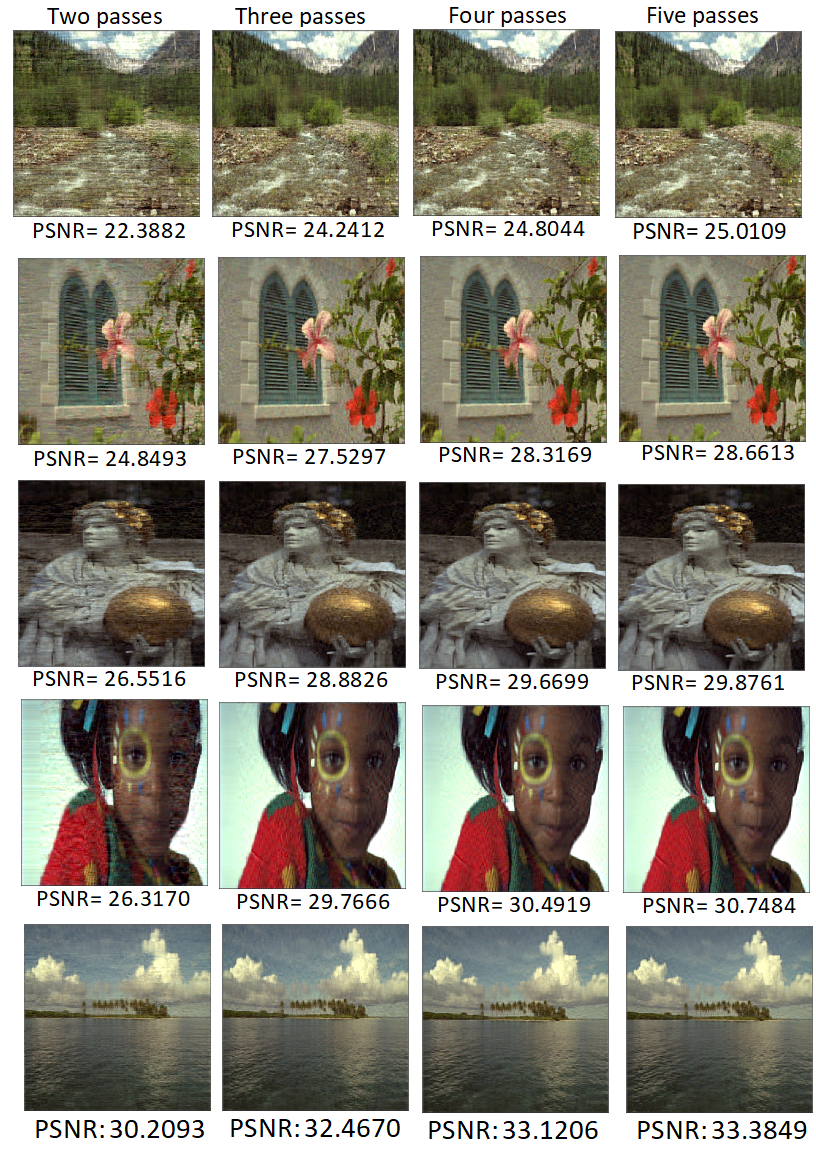}
\caption{\small{Compression results for four benchmark images and different passes.}}\label{Res_1}
\end{center}
\end{figure}

\begin{figure}
\begin{center}
\includegraphics[width=1.1\linewidth]{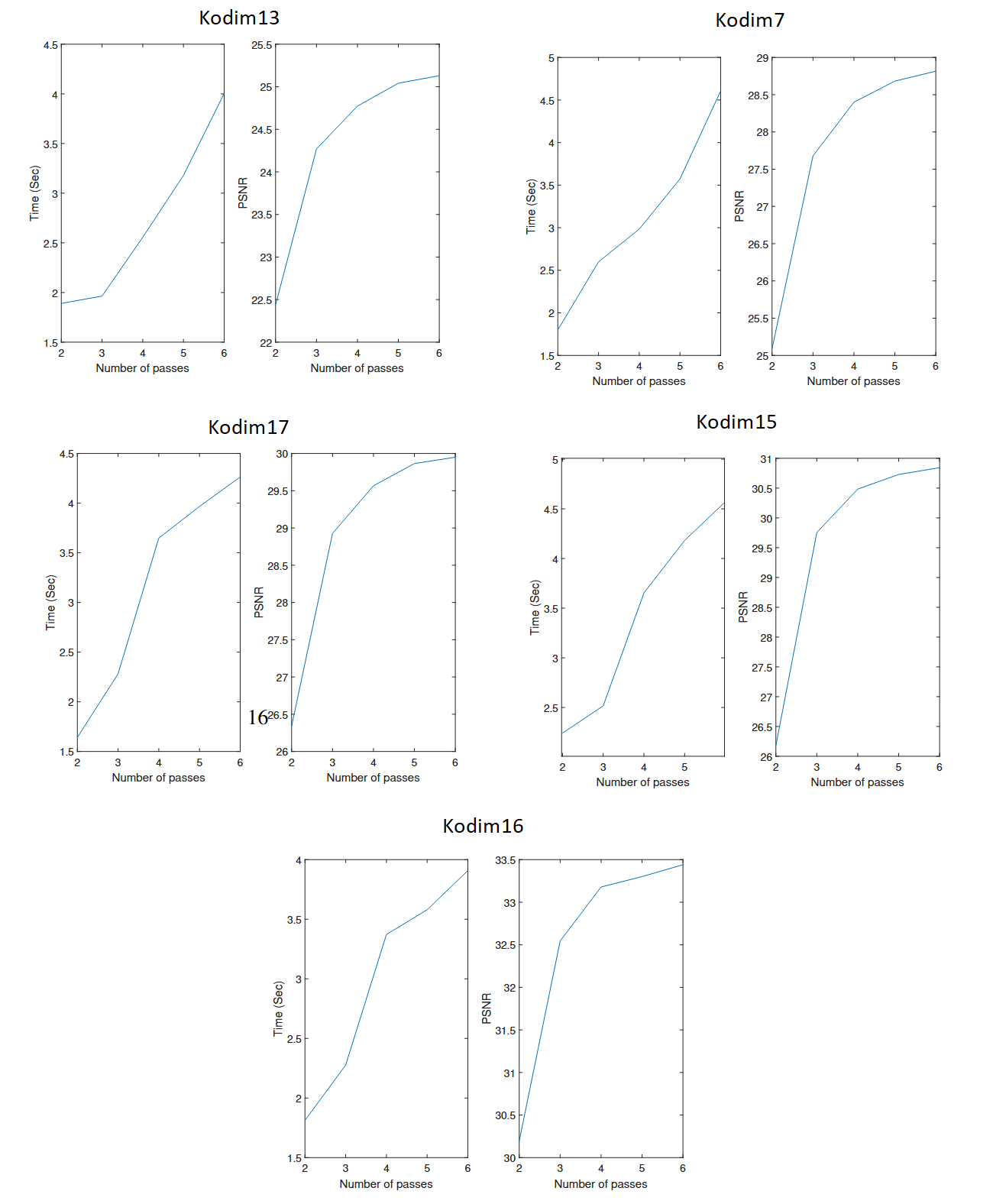}
\caption{\small{The time and PSNR comparisons for different benchmark images using different passes.}}\label{Res_2}
\end{center}
\end{figure}

\begin{table}
\begin{center}
\caption{Comparing the running time and the PSNR achieved by the proposed algorithms and baselines for Example \ref{exa_1}. The results are for the {matrix rank} $R=30$. For Algorithms 1 and 3, four passes were used ($q=1$) while for Algorithms 2 and 4, three passes were used ($v=3$). }\label{Table_Salman}
\vspace{0.2cm}
{\smaller
\begin{tabular}{||c| c | c ||} 
 \multicolumn{3}{c}{Kodim13 }\\
 \hline
Algorithms  & Running Time (Seconds) & PSNR \\
 \hline\hline
 Algorithm 1& 2.6&  24.5\\ 
  Algorithm 2& {\bf 1.9}& 23.6\\ 
 Algorithm 3& 3.5&   {\bf 24.6}\\
 Algorithm 4& 2.3&   24.5\\
 \hline
 \multicolumn{3}{c}{Kodim7 }\\\hline
 Algorithms  & Running Time (Seconds) & PSNR \\
 \hline\hline
 Algorithm 1&  2.9&  28.2\\ 
  Algorithm 2&  {\bf 1.9}&  27.40\\ 
 Algorithm 3& 3.6&  {\bf 28.43}\\
 Algorithm 4& 3.1&  27.40\\
 \hline
  \multicolumn{3}{c}{Kodim17 }\\\hline
 Algorithms  & Running Time (Seconds) & PSNR \\
 \hline\hline
 Algorithm 1&  3.6&  {\bf 29.3}\\ 
  Algorithm 2&  {\bf 2.1}&  28.4\\ 
 Algorithm 3& 4.1&  29.3\\
 Algorithm 4& 3.3&  28.3\\
 \hline
   \multicolumn{3}{c}{Kodim15 }\\\hline
 Algorithms  & Running Time (Seconds) & PSNR \\
 \hline\hline
 Algorithm 1&  3.5&  30.7\\ 
  Algorithm 2&  {\bf 2.4}&  29.6\\ 
 Algorithm 3& 3.9&  {\bf 30.8}\\
 Algorithm 4& 3.6&  29.7\\
 \hline
  \multicolumn{3}{c}{Kodim16 Image}\\\hline
 Algorithms  & Running Time (Seconds) & PSNR \\
 \hline\hline
 Algorithm 1&  3.5&  {\bf 33.3}\\ 
  Algorithm 2&  {\bf 2.3}&  32.5\\ 
 Algorithm 3& 4.2&  33.3\\
 Algorithm 4& 3.2&  32.6\\
 \hline
\end{tabular}
}
\end{center}
\end{table}

\end{exa}

\begin{exa}\label{exa_2}
(Applications in image completion and image super-resolution) 
We now evaluate the proposed randomized pass-efficient algorithm on two additional tasks: image completion with random missing pixels, and image super-resolution (SR) with structured missing patterns, as introduced in Section~\ref{Sec:appl}.

\textit{Image Completion.}  
We apply the quaternion matrix completion method to images ``Kodim11'', ``Kodim13'', ``Kodim16'', and ``Kodim24'', all resized to $256 \times 256 \times 3$ for consistency. For each image, 70\% of the pixels were removed uniformly at random. The reconstruction was performed using the low-rank approximation with target rank $R=30$ and two passes of the data. To improve visual coherence, a Gaussian smoothing filter with parameter 0.6 was applied after each iteration. The reconstructed images are shown in Figure~\ref{img_recov}, demonstrating high perceptual quality and structural consistency despite the large proportion of missing data. {The proposed method (Algorithm \ref{ALg_2}), was used for all experimental evaluations in this example.}

\textit{Image Super-Resolution.}  
We next consider ``Kodim2'', ``Kodim15'', ``Kodim16'', and ``Kodim21'', which were downsampled by a factor of 4 and then upsampled back to their original resolution. This process simulates structured missing data by removing regular blocks of pixels: specifically, we retain a fixed block of columns/rows and omit the next, repeating this pattern throughout. The recovery was performed using the same completion algorithm with rank $R=20$, two passes, and Gaussian smoothing (parameter 0.6). Figure~\ref{sr_img} displays the results. The proposed method (Algorithm \ref{ALg_2}), was used for all experimental evaluations in this example.

These experiments confirm that the proposed pass-efficient randomized algorithm is effective in both random and structured data completion scenarios, making it suitable for real-world image restoration and super-resolution applications.

\begin{figure}
\begin{center}
\includegraphics[width=0.9\linewidth]{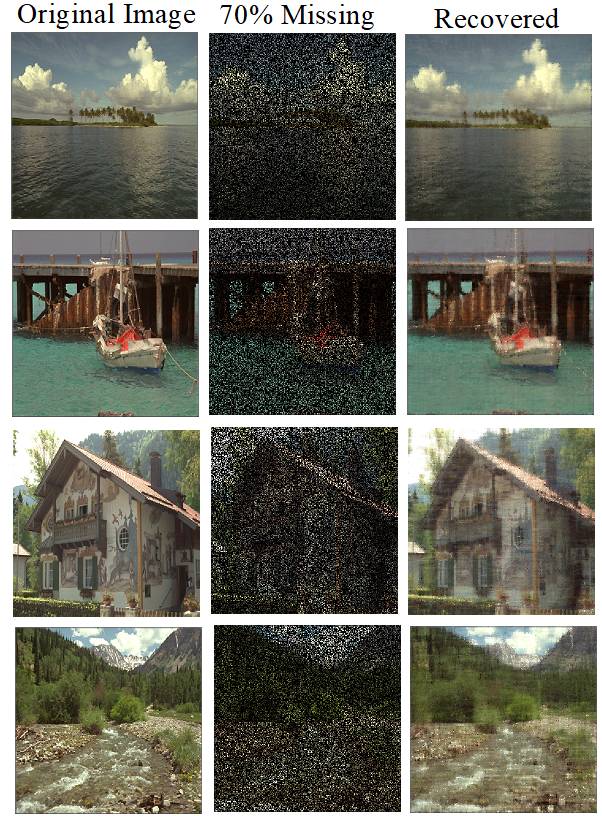}
\caption{\small{Recovered images with 70\% missing pixels. Original image (left), image with missing pixels (middle image) and recovered image (right image).}}\label{img_recov}
\end{center}
\end{figure}

\begin{figure}
\begin{center}
\includegraphics[width=0.9\linewidth]{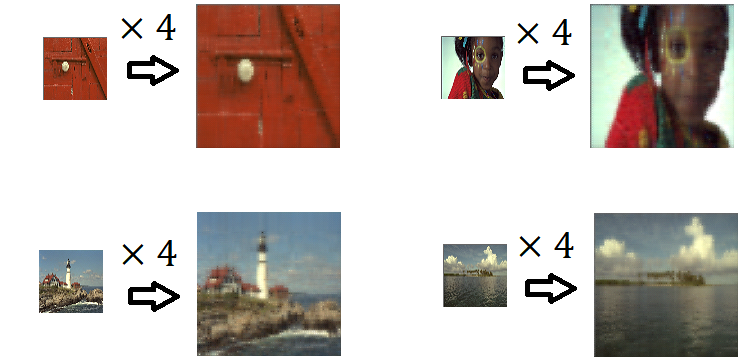}
\caption{\small{Super resolution results for four images .}}\label{sr_img}
\end{center}
\end{figure}

\end{exa}

\begin{exa}\label{exa_3}
(Application in deep learning) 
In this experiment, we demonstrate how the proposed quaternion matrix completion method can enhance the robustness of deep neural networks (DNNs) on an instance image segmentation task—a key problem in computer vision.
Unlike semantic segmentation, which assigns class labels to pixels without distinguishing between individual objects, instance segmentation assigns unique identifiers to each object instance; see~\cite{minaee2021image} for a detailed overview. We use the YOLOv8 Instance Segmentation model~\cite{ultralytics_yolov8} for evaluation.

Figure~\ref{ins_seg_1} illustrates the experimental setup. The top row shows the original image (leftmost) and three degraded variants with subtle perturbations. The bottom row presents the corresponding instance segmentation outputs. As seen, the YOLOv8 model performs well on the original image but exhibits sensitivity to even mild corruptions. For example, in the second image, the network misclassifies a handbag as part of the dog; in the third and fourth images, the model fails to detect the dog entirely due to minor distortions near the eye or body.

To mitigate this degradation, we apply the proposed randomized quaternion completion algorithm (Section~\ref{Sec:appl}) to recover the corrupted images before passing them to the segmentation model. Please note that the proposed method (Algorithm \ref{ALg_2}), was used for all experimental evaluations in this example. The results, shown alongside the original segmentation, indicate that our method restores key structural features, enabling the DNN to produce outputs closely matching those of the uncorrupted image. This demonstrates the potential of quaternion completion as a lightweight and effective defense mechanism against input perturbations in deep learning pipelines.

\begin{figure}
\begin{center}
\includegraphics[width=0.24\linewidth]{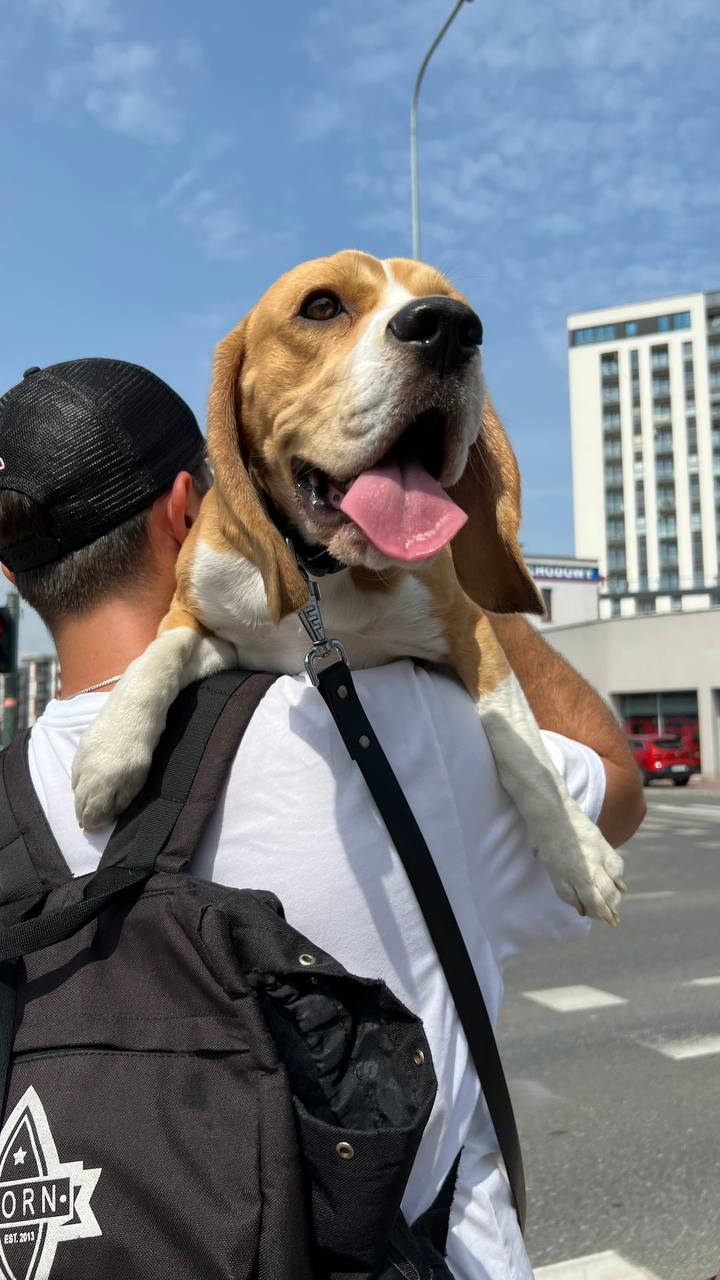}
\includegraphics[width=0.24\linewidth]{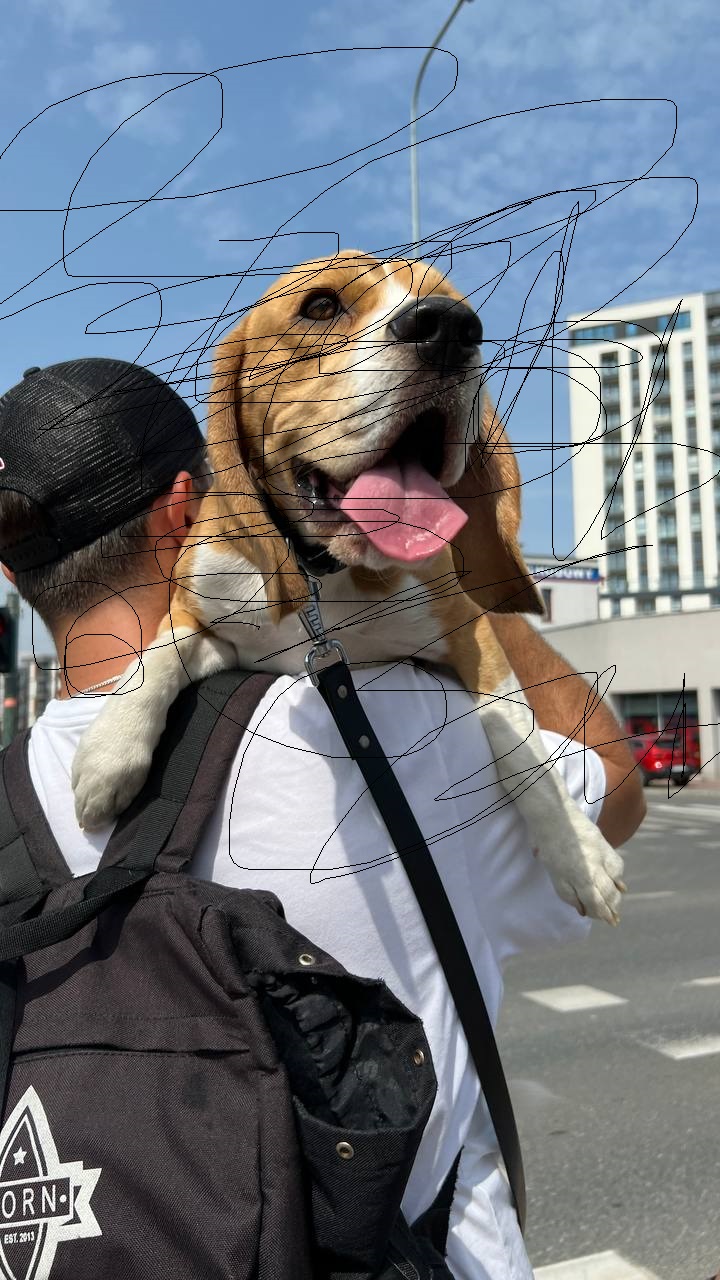}
\includegraphics[width=0.24\linewidth]{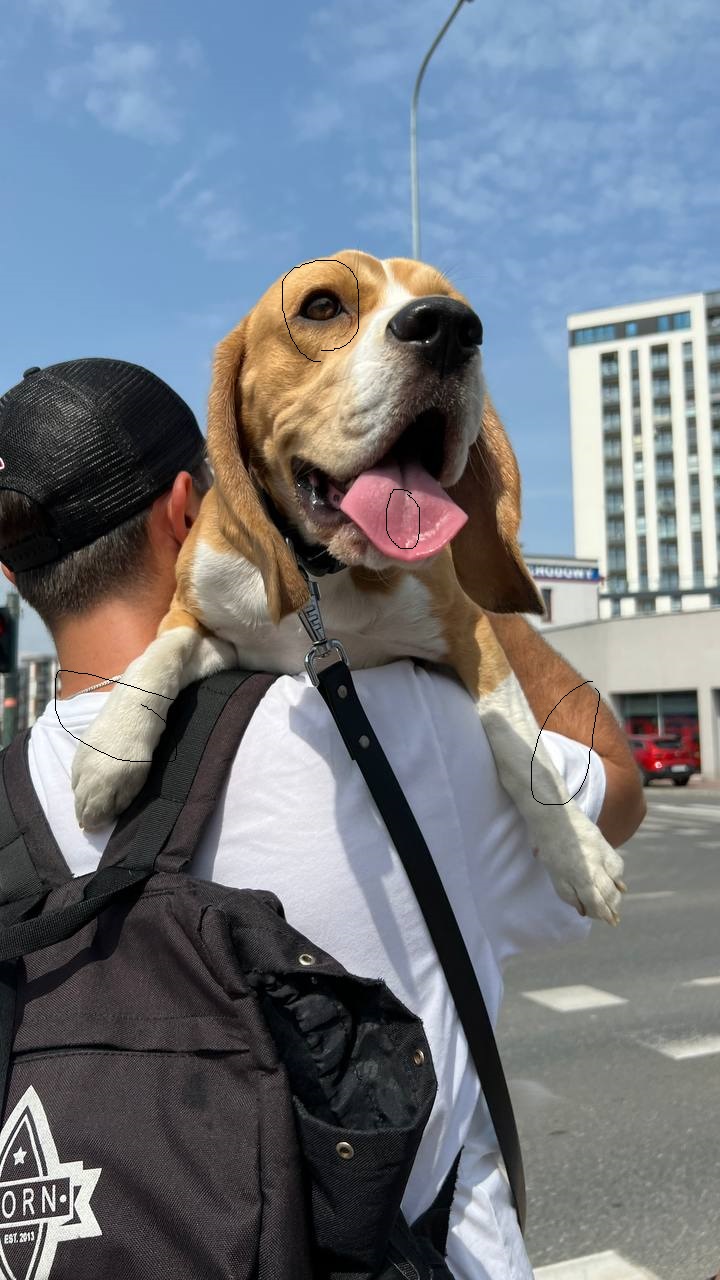}
\includegraphics[width=0.24\linewidth]{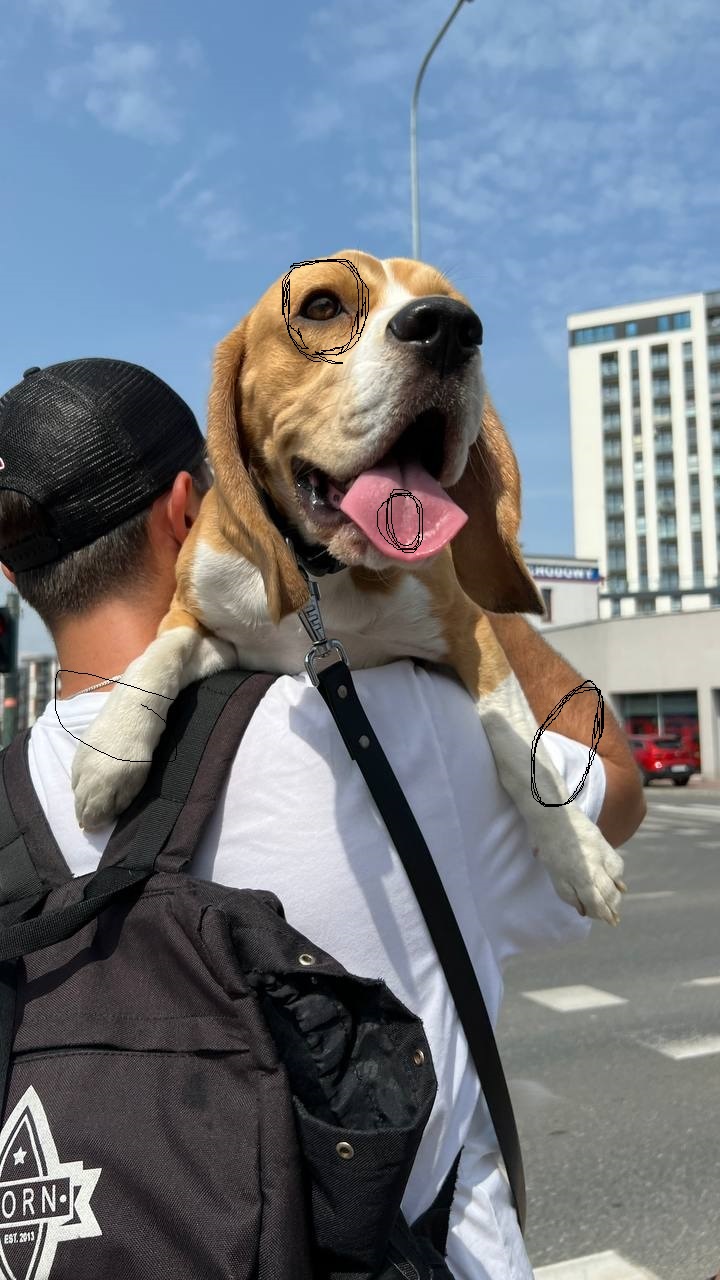}
\includegraphics[width=0.24\linewidth]{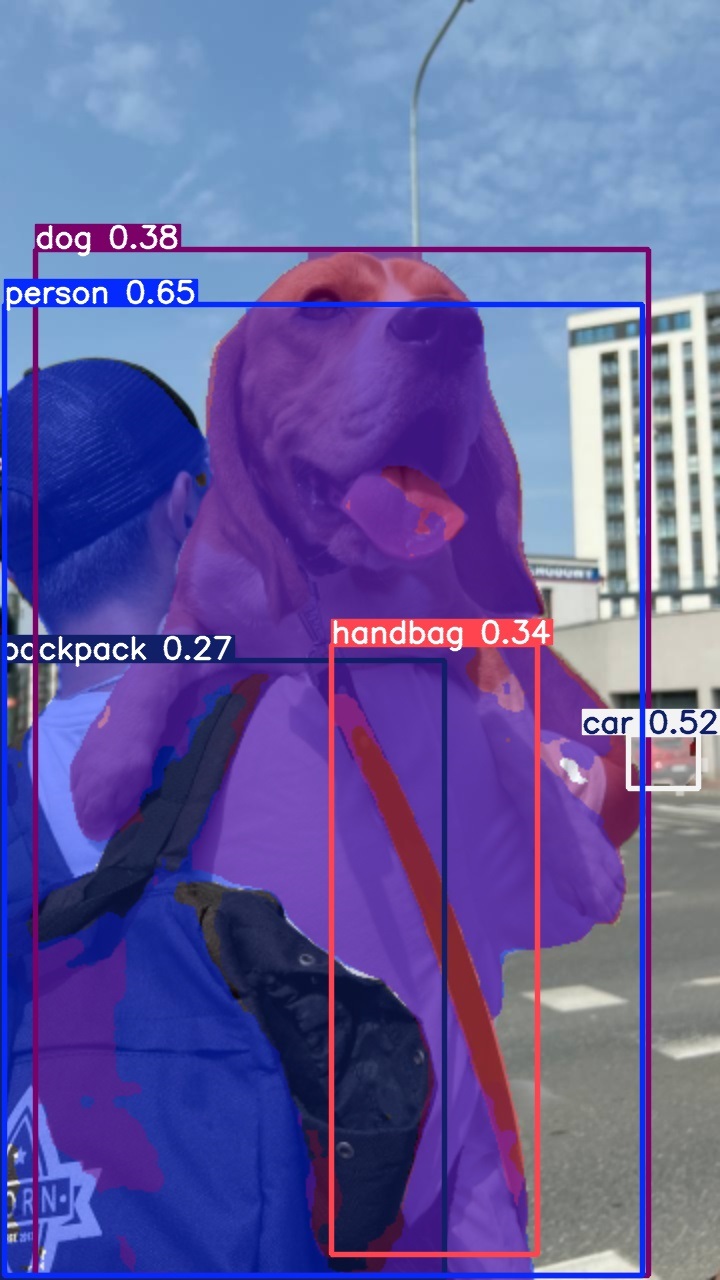}
\includegraphics[width=0.24\linewidth]{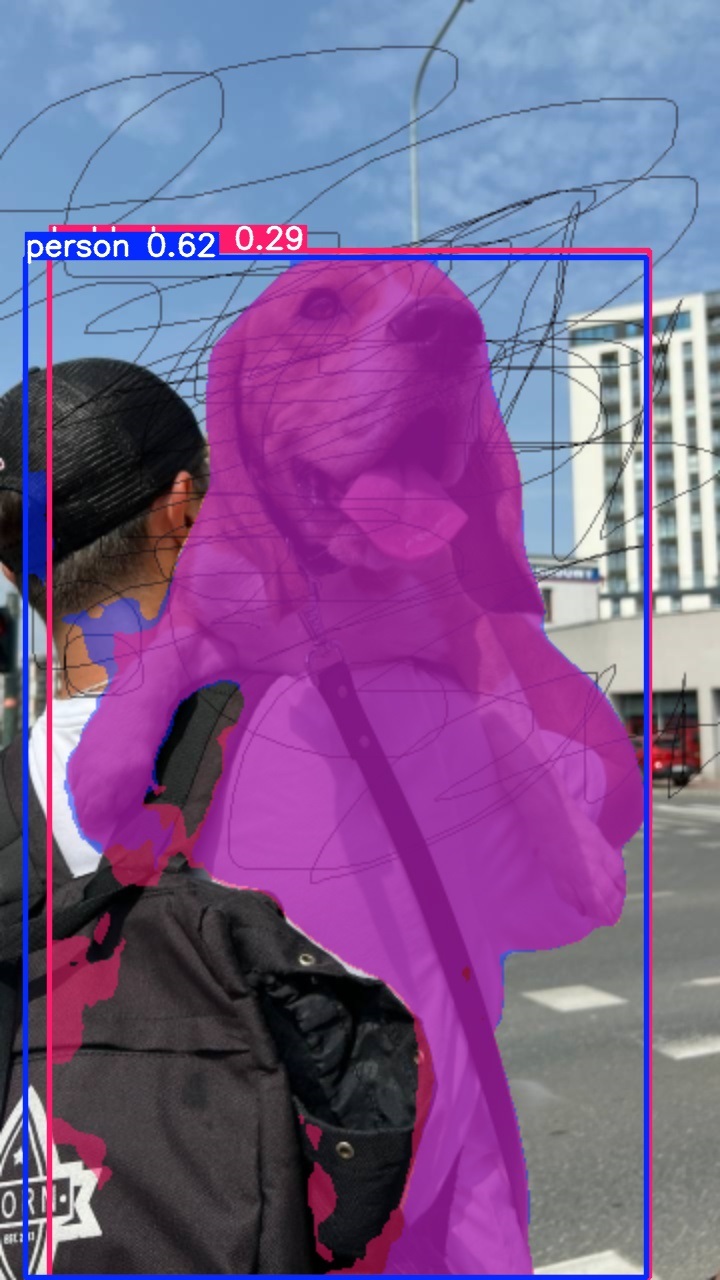}
\includegraphics[width=0.24\linewidth]{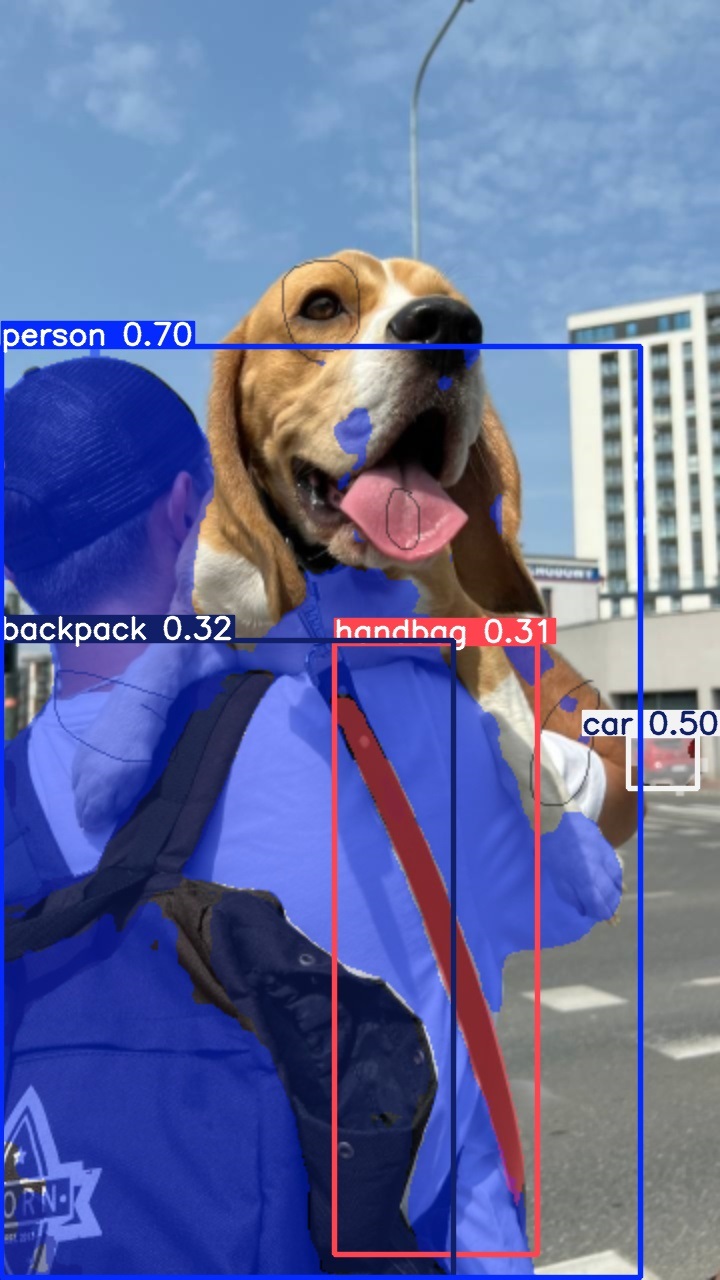}
\includegraphics[width=0.24\linewidth]{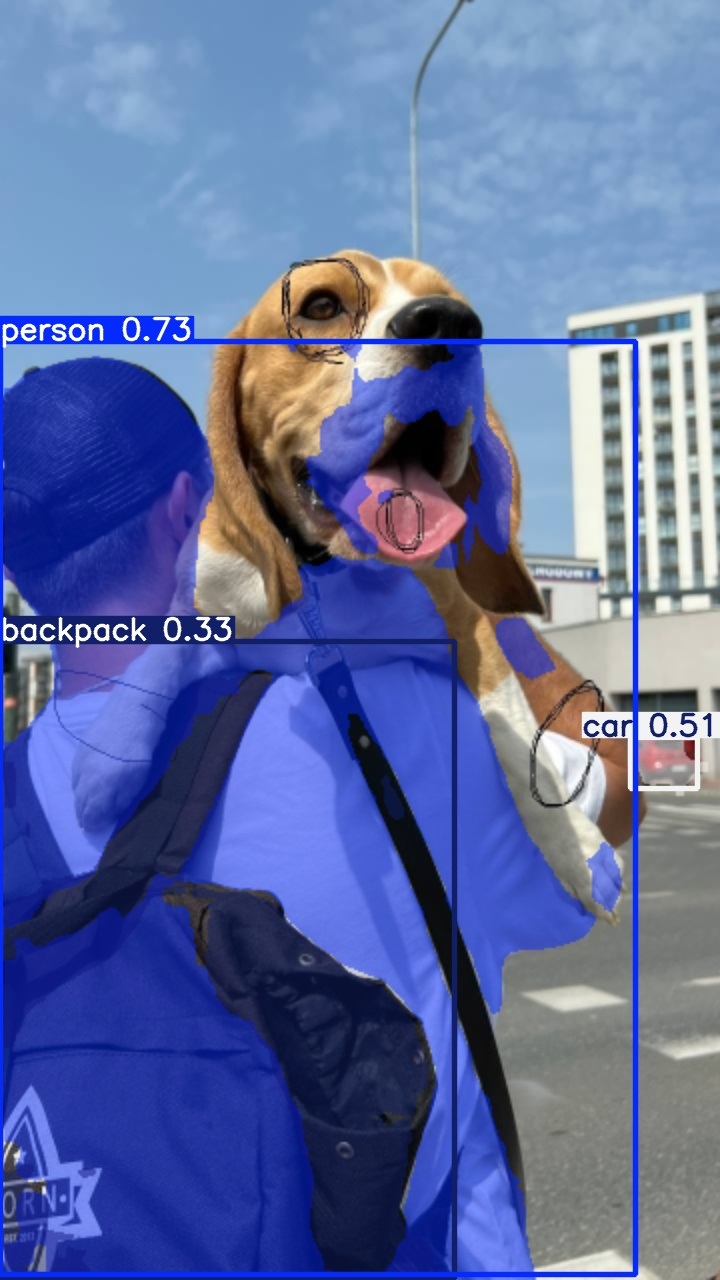}
\caption{\small{The upper left image is the original image, and the next three images are the degraded images. The images in the bottoms are the instance segmentation corresponding to the upper images.}}\label{ins_seg_1}
\end{center}
\end{figure}

\begin{figure}
\begin{center}
\includegraphics[width=0.24\linewidth]{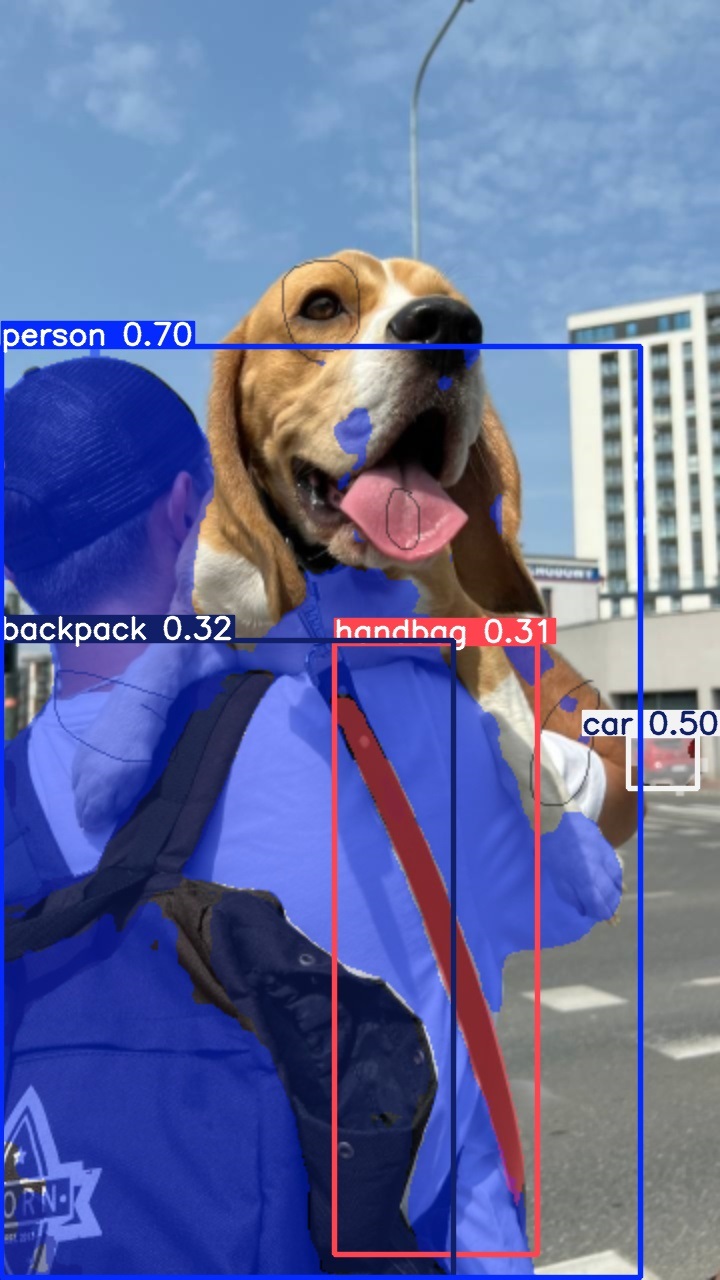}
\includegraphics[width=0.24\linewidth]{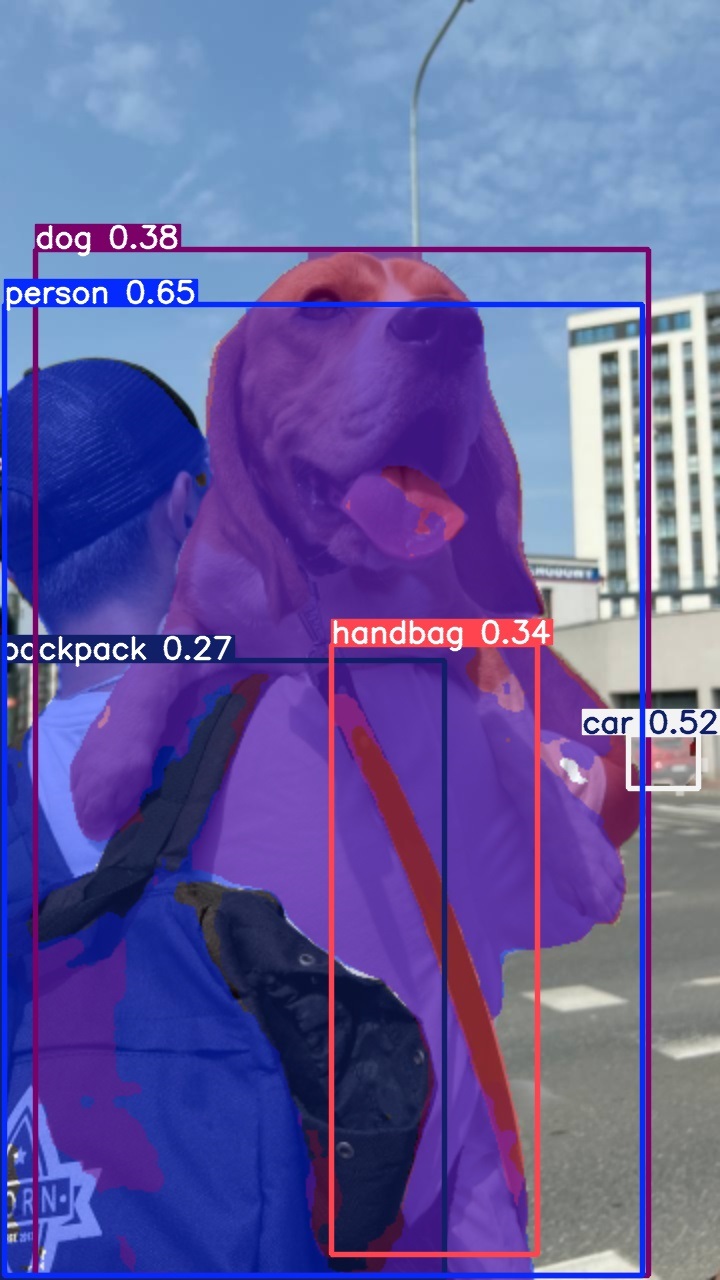}
\caption{\small{The instance segmentation of the images before (left image) and after (right image) the reconstruction process.}}\label{ins_seg_2}
\end{center}
\end{figure}

\end{exa}

{
\begin{exa}{(Random quaternion matrices)}\label{exa_rq}
This example provides an example of applying the proposed randomized algorithm to random low-rank quaternion matrices of larger sizes compared to the previous examples. 
We generate a random low-rank quaternion matrix ${\bf A} \in \mathbb{Q}^{m \times n}$ 
with a specified rank $r$ (where $r \ll \min(m,n)$) as
\[
{\bf A} = {\bf B}\,{\bf C},
\]
where ${\bf B}\in\mathbb{Q}^{m\times r}$ and ${\bf C}\in\mathbb{Q}^{r\times n}$ are 
Gaussian random quaternion matrices.
In our experiments we used $m=n=500,1000,1500,\ldots,4000$ and the rank $R=50$. Please note that a quaternion matrix of size $4000 \times 4000$ requires approximately 0.47 GB of memory for storage. Figure \ref{rand_ex} reports the computing time and relative error for the deterministic SVD, the classical randomized algorithm (Algorithm \ref{ALg_1}) with $q=1,2$, and the proposed pass-efficient randomized algorithm (Algorithm \ref{ALg_2}) using two and three passes. For clarity, we omit the results from Algorithm \ref{ALg_block_flix}, as they were nearly identical to those of Algorithm \ref{ALg_2} both relative error and computing time. The experiments confirm that the proposed algorithms achieve satisfactory accuracy with significantly faster computation times than the baseline methods.
\begin{figure}
\begin{center}
\includegraphics[width=0.49\linewidth]{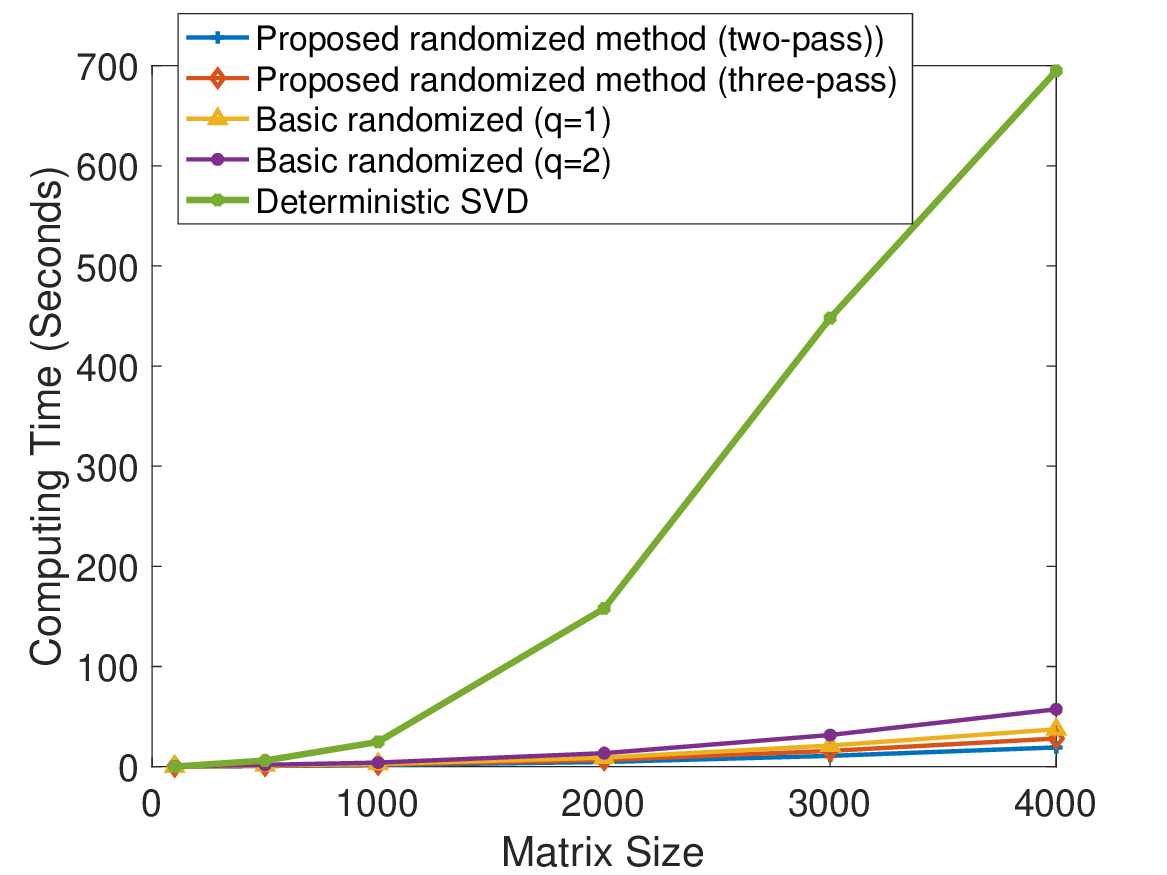}
\includegraphics[width=0.49\linewidth]{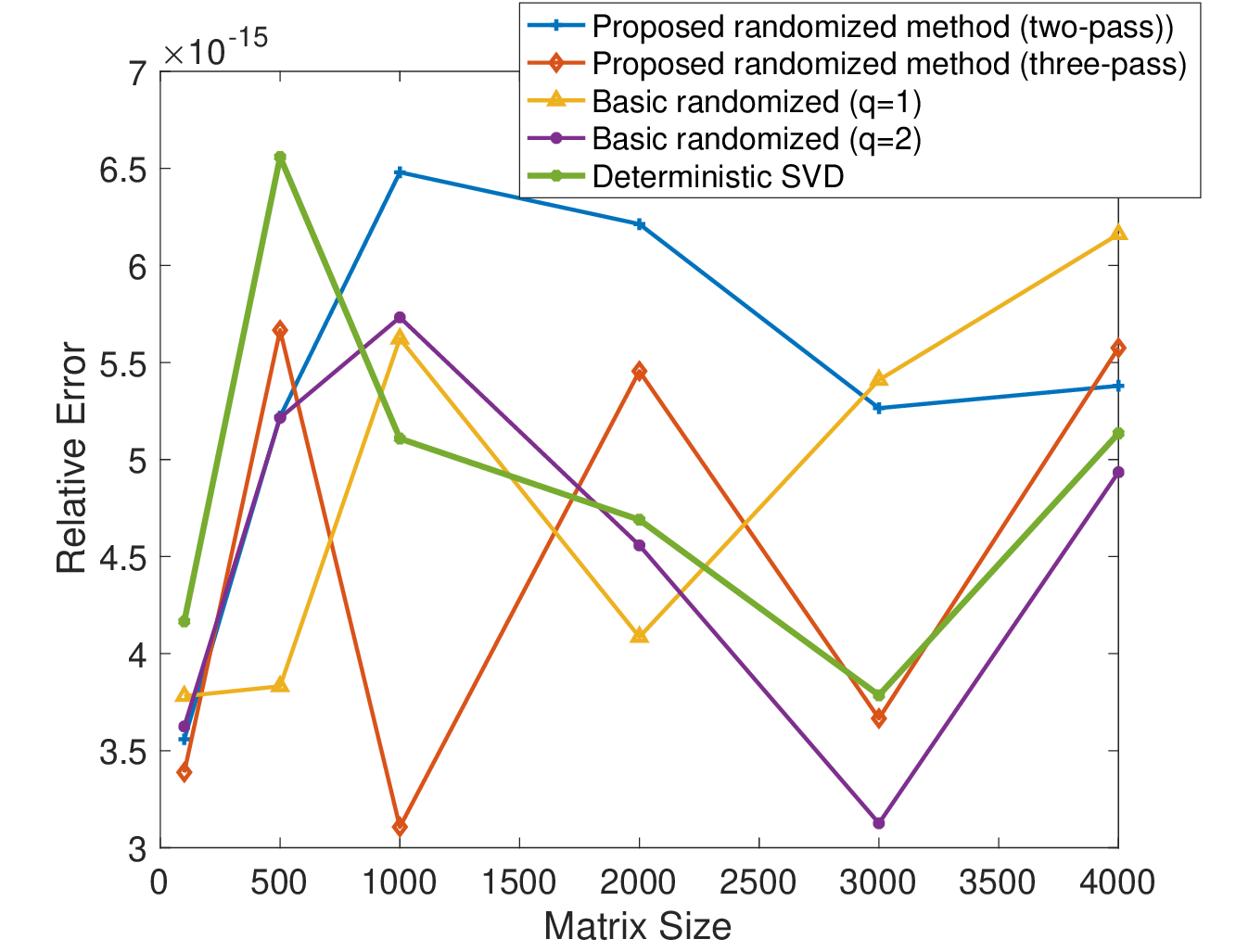}
\caption{\small{ The simulation results for random low-rank quaternion matrices.}}\label{rand_ex}
\end{center}
\end{figure}
\end{exa}}

\begin{exa}
{(Application to scientific data)} 
In this example, we compress scientific data generated from a simulation of a Lorenz-type chaotic system \cite{jia2021structure,chang2024randomized}. The Lorenz attractor, a three-dimensional nonlinear system originally developed to model atmospheric turbulence \cite{strogatz2024nonlinear}, is described by the following system of coupled differential equations:
\begin{align*}
    \frac{dx}{dt} &= \sigma (y - x), \\
    \frac{dy}{dt} &= x (\rho - z) - y, \\
    \frac{dz}{dt} &= x y - \beta z,
\end{align*}
where $\sigma,\beta,\rho>0.$ To achieve chaotic behavior in the Lorenz attractor, the parameters were set to $\sigma = 10$, $\beta = 8/3$, and $\rho = 28$ (the same parameters as  used in \cite{jia2021structure}). The system of coupled differential equations (5.6) was solved numerically using MATLAB's \texttt{ODE45} solver with the syntax \texttt{ODE45($f(t, [x, y, z])$, $[0, T]$, $[1, 1, 1]$)}, where $T > 0$ represents the simulation time.
In this example, the output vector $\mathbf{y}(t)$ is defined by the state of the Lorenz attractor,
\[
\mathbf{y}(t) = y_r(t)\mathbf{i} + y_g(t)\mathbf{j} + y_b(t)\mathbf{k},
\]
where $y_r(t)$, $y_g(t)$, and $y_b(t)$ are its component solutions. The input vector $\mathbf{x}(t)$ is a time-delayed and noisy version of the state,
\[
\mathbf{x}(t) = y_r(t - 1)\mathbf{i} + y_g(t - 1)\mathbf{j} + y_b(t - 1)\mathbf{k} + \mathbf{n}(t),
\]
with $\mathbf{n}(t)$ representing additive random noise. Following \cite{jia2021structure}, this process is modeled by a linear system featuring a quaternion coefficient matrix of size $N \times N$, where $N$ is a function of $T$. In our computations, we treat the dimensions generated by $T = 10, 20, 30, 40$ (yielding $N = 425, 957, 1513, 2105$, respectively) as scientific data.

The singular values for the cases $425\times 425$ and $957\times 957$ are shown in Figure \ref{sd_ex}. Given the relatively fast decay of singular values, the coefficient matrix exhibits a low-rank structure. To leverage this with a large-scale matrix, we set \(T=40\), generating a quaternion matrix of size \(2105 \times 2105\). The performance of our proposed randomized algorithm was evaluated by computing low-rank approximations at ranks \(R = 50, 100, 200, 300\). Figure~\ref{sd_ex_r} compares the resulting computational time and relative error against established baselines.
 The experimental results consistently demonstrate a substantial reduction in computational time for our methods while maintaining solution quality comparable to existing approaches.
\begin{figure}
\begin{center}
\includegraphics[width=0.49\linewidth]{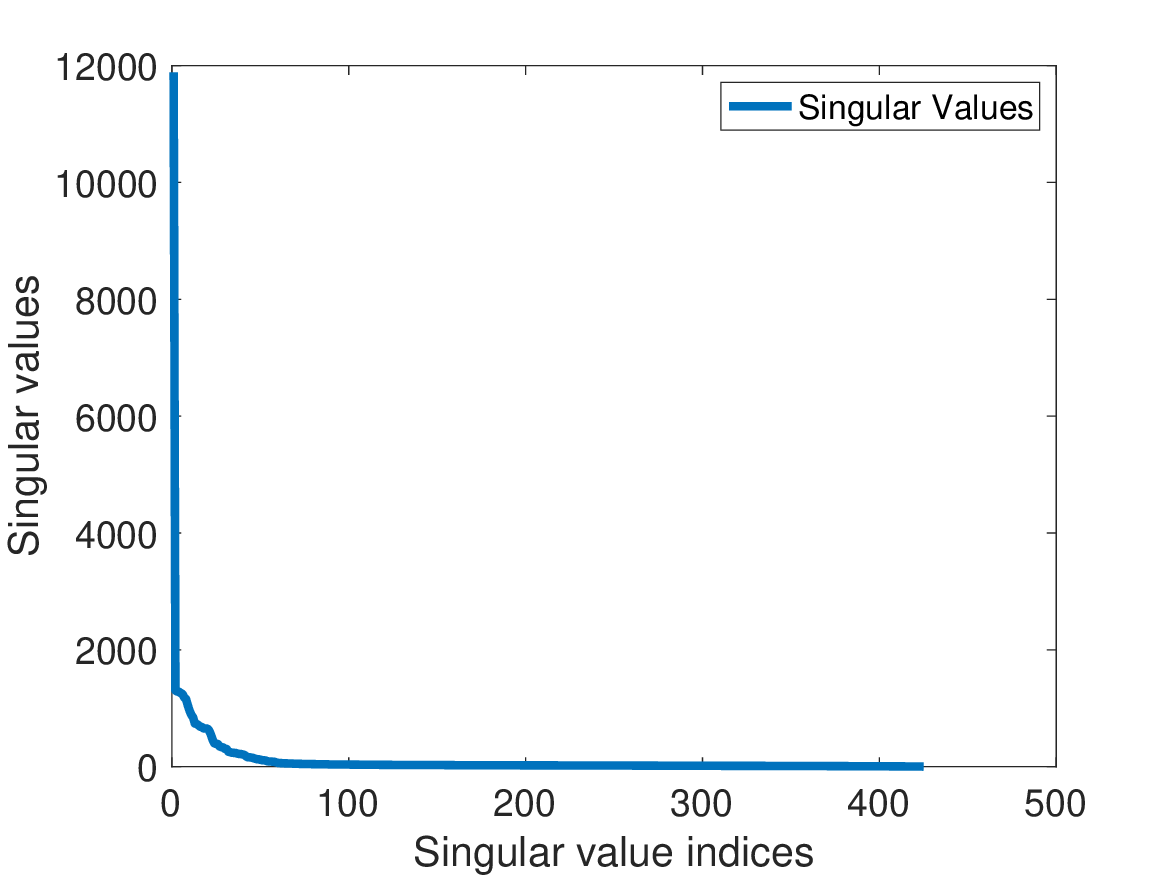}
\includegraphics[width=0.49\linewidth]{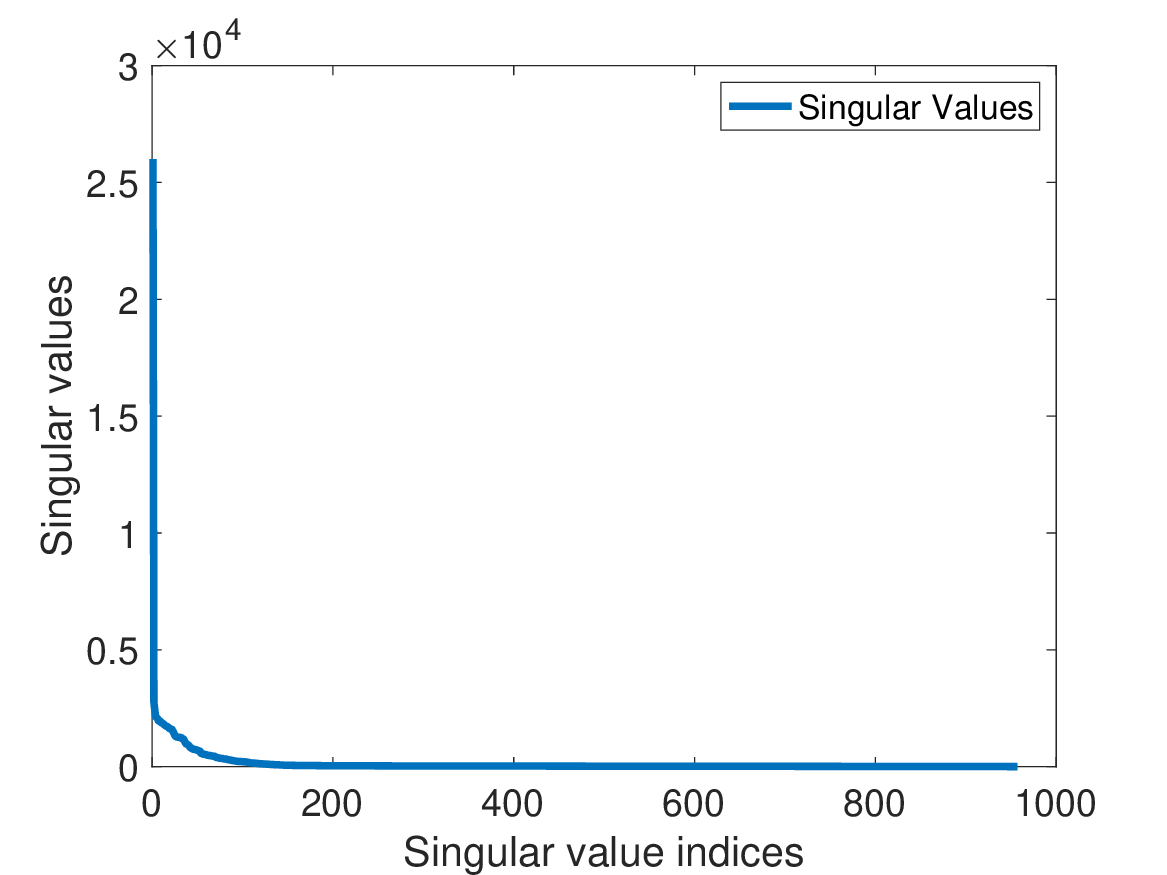}
\caption{\small{Singular values of the scientific data in Lorenz-type chaotic system (Left: a quaternion of size $425\times 425$) and (Right: a quaternion of size $957\times 957$).}}\label{sd_ex}
\end{center}
\end{figure}
\begin{figure}
\begin{center}
\includegraphics[width=0.49\linewidth]{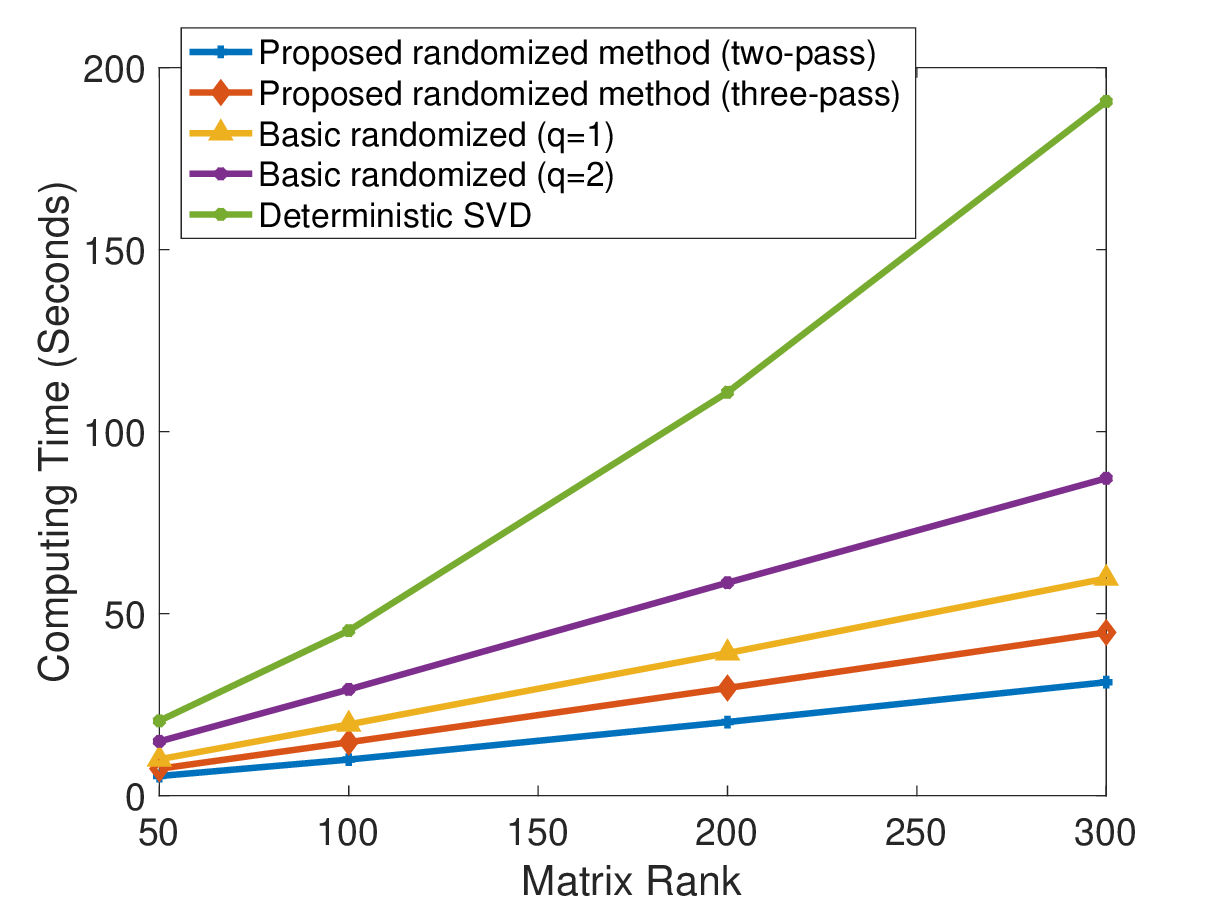}
\includegraphics[width=0.49\linewidth]{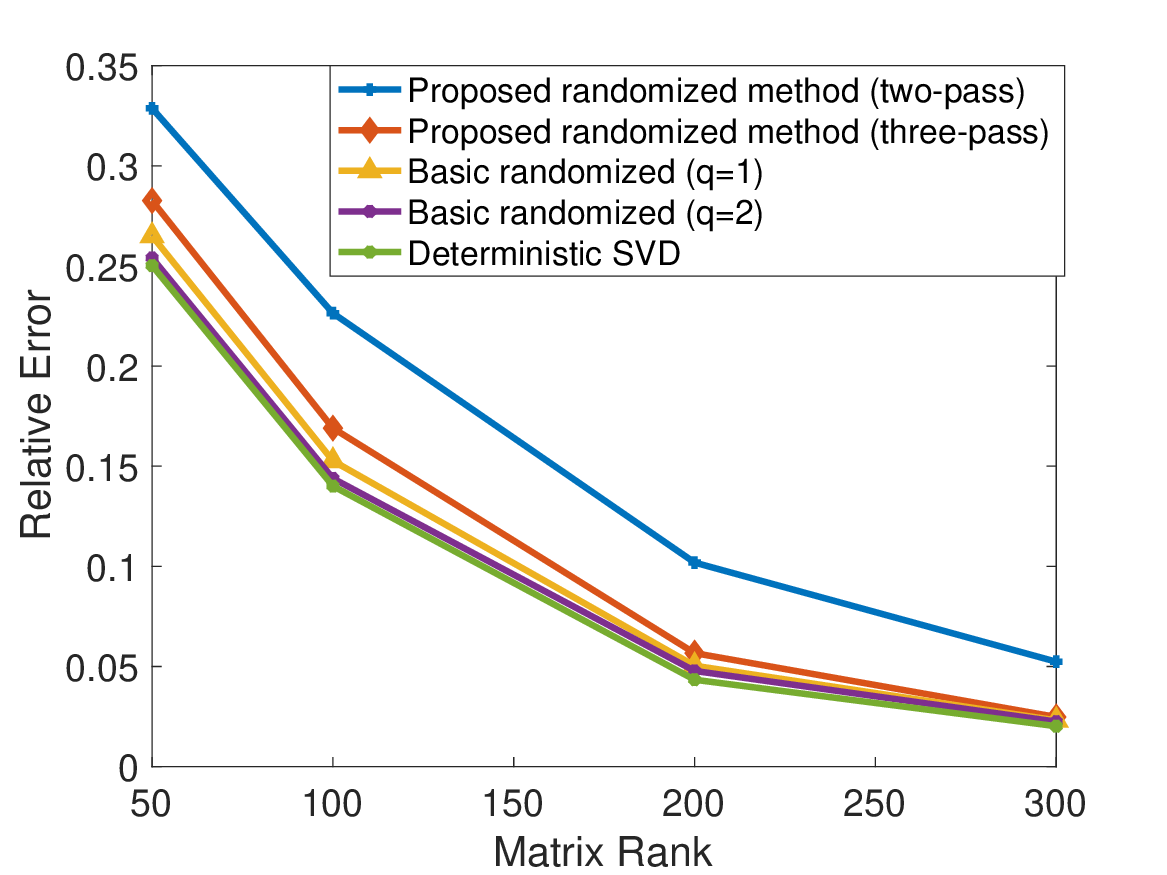}
\caption{\small{The simulation results for the scientific data matrix.}}\label{sd_ex_r}
\end{center}
\end{figure}
\end{exa}

\color{black}

\section{Conclusion and future work}\label{sec:conclu}

In this paper, we introduced a family of pass-efficient randomized algorithms for low-rank approximation of quaternion matrices. Our main contribution lies in enabling users to flexibly control the number of matrix passes---a key constraint in modern large-scale computations---while maintaining high approximation quality. The proposed algorithms include both classical randomized subspace iteration and its extension to block Krylov methods, all tailored for the quaternion setting. We provided theoretical analysis establishing spectral norm error bounds, showing that the expected approximation error decays exponentially with the number of passes. These theoretical results were substantiated by a suite of numerical experiments, which demonstrated the practical efficiency of our methods across various applications, including image compression, matrix completion, super-resolution, and robust deep learning pipelines. Beyond validating our algorithms in realistic settings, our work opens up several promising directions for future research:

\begin{itemize}
    \item \textbf{Generalization to broader algebraic structures:} We plan to extend our randomized framework to matrices over split-quaternions~\cite{ablamowicz2020moore} and Clifford algebras~\cite{cao2022moore}, which naturally arise in signal processing and geometric computing.

    \item \textbf{Structure-preserving implementations:} While all implementations in this work were carried out using quaternion arithmetic directly, we are actively investigating structure-preserving formulations (e.g., real block matrix representations) to further reduce computation time and memory usage.

    \item \textbf{Acceleration of higher-order decompositions:} Our methods can serve as efficient subroutines for computing the recently introduced Quaternion Tensor SVD (QTSVD)~\cite{pan2024block}. Since QTSVD relies on repeated QSVD computations, integrating our pass-efficient strategies promises to significantly improve its efficiency.

    \item \textbf{ Proper Orthogonal Decomposition (POD):} Investigating the application of our method to model order reduction, specifically using Proper Orthogonal Decomposition \cite{chatterjee2000introduction}.
\end{itemize}

We believe that the proposed algorithms offer a solid foundation for efficient low-rank modeling in quaternionic and hypercomplex domains.

\section*{Acknowledgment} The work was supported by the Ministry of Economic Development of the Russian Federation under Agreement No. 139-10-2025-034 dd. 19.06.2025, IGK 000000C313925P4D0002.

 \bibliographystyle{elsarticle-num} 
 \bibliography{cas-refs}





\end{document}